\documentclass[a4paper,10pt]{article}
\usepackage[utf8]{inputenc}

\usepackage{amsmath,amssymb,amsthm}
\usepackage{amscd}
\usepackage{mathrsfs}
\usepackage[usenames,dvipsnames]{color}
\usepackage{pstricks}
\usepackage{hyperref}
\usepackage{graphicx}
\usepackage{verbatim}
\usepackage{geometry}
\usepackage{extarrows}
\usepackage{tabularx}
\usepackage{slashed}
\usepackage{xparse}
\usepackage{enumitem}

\usepackage{xspace}

\def\shuffle{{\sqcup\mathchoice{\mkern-3mu}{\mkern-3mu}{\mkern-3.2mu}{\mkern-3.8mu}\sqcup}} 
\def\stuffle{{\sqcup\mathchoice{\mkern-12.5mu}{\mkern-12.5mu}{\mkern-8.7mu}{\mkern-8.5mu} - \mathchoice{\mkern-12.5mu}{\mkern-12.5mu}{\mkern-8.5mu}{\mkern-8.7mu}\sqcup}}

\usepackage{tikz}
\tikzstyle{every picture}=[level distance = 8mm, baseline=-0.5ex]
\tikzstyle{prop}=[shape=circle,minimum size=6mm, draw=black!80, fill=green!30]

\usetikzlibrary{arrows,matrix}
\usetikzlibrary{decorations.pathmorphing}

\newcommand{\C}{\mathbb{C}}
\newcommand{\R}{\mathbb{R}}
\newcommand{\Q}{\mathbb{Q}}
\newcommand{\Z}{\mathbb{Z}}
\newcommand{\N}{\mathbb{N}}

\newtheorem{thm}{Theorem}[section]
\newtheorem*{thm*}{Theorem}
\newtheorem{lem}[thm]{Lemma}
\newtheorem{coro}[thm]{Corollary}

\newtheorem{prop}[thm]{Proposition}

\newtheorem{defnprop}[thm]{Definition-Proposition}

\theoremstyle{definition}
\newtheorem{defn}[thm]{Definition}
\newtheorem{rk}[thm]{Remark}
\newtheorem{ex}[thm]{Example}
\newtheorem{coex}[thm]{Counterexample}

\newtheorem{innercustomgeneric}{\customgenericname}
\providecommand{\customgenericname}{}
\newcommand{\newcustomtheorem}[2]{%
  \newenvironment{#1}[1]
  {%
   \renewcommand\customgenericname{#2}%
   \renewcommand\theinnercustomgeneric{##1}%
   \innercustomgeneric
  }
  {\endinnercustomgeneric}
}

\newcustomtheorem{customthm}{Theorem}

\newcommand {\fraks}{{\mathfrak {s}}}
\newcommand {\frakt}{{\mathfrak {t}}}

\newcommand {\frakS}{{\mathfrak {S}}}

\newcommand {\calf}{{\mathcal{F}}}

\newcommand {\cali}{\mathcal{I}}

\newcommand {\calp}{{\mathcal {P}}}
\newcommand {\calr}{{\mathcal {R}}}
\newcommand {\cals}{{\mathcal {S}}}

\newcommand {\calw}{{\mathcal {W}}}
\newcommand {\calz}{{\mathcal {Z}}}

\newcommand{\wP}{{\widehat P}}
\newcommand{\wPw}{{\widehat P}_\calw}

\newcommand{\Ker}{\text{Ker}}

\newcommand{\tdun}[1]{\begin{picture}(10,5)(-2,-1)
\put(0,0){\circle*{2}}
\put(2,-2){\tiny #1}
\end{picture}}

\newcommand{\tddeux}[2]{\begin{picture}(12,5)(0,-1)
\put(3,0){\circle*{2}}
\put(3,0){\line(0,1){5}}
\put(3,5){\circle*{2}}
\put(4,-2){\tiny #1}
\put(4,4){\tiny #2}
\end{picture}}

\newcommand{\tdtroisun}[3]{\begin{picture}(20,12)(-5,-1)
\put(3,0){\circle*{2}}
\put(-0.65,0){$\vee$}
\put(6,7){\circle*{2}}
\put(0,7){\circle*{2}}
\put(5,-2){\tiny #1}
\put(6,5){\tiny #2}
\put(-5,8){\tiny #3}
\end{picture}}

\newcommand{\tdquatredeux}[4]{\begin{picture}(20,20)(-5,-1)
\put(3,0){\circle*{2}}
\put(-.65,0){$\vee$}
\put(6,7){\circle*{2}}
\put(0,7){\circle*{2}}
\put(0,14){\circle*{2}}
\put(0,7){\line(0,1){7}}
\put(5,-2){\tiny #1}
\put(9,5){\tiny #2}
\put(-5,5){\tiny #3}
\put(-5,12){\tiny #4}
\end{picture}}

\newcommand{\tdquatretrois}[4]{\begin{picture}(20,20)(-5,-1)
\put(3,0){\circle*{2}}
\put(-.65,0){$\vee$}
\put(6,7){\circle*{2}}
\put(0,7){\circle*{2}}
\put(6,14){\circle*{2}}
\put(6,7){\line(0,1){7}}
\put(5,-2){\tiny #1}
\put(9,5){\tiny #2}
\put(-5,5){\tiny #4}
\put(9,12){\tiny #3}
\end{picture}}

\newcommand{\tdquatrequatre}[4]{\begin{picture}(20,14)(-5,-1)
\put(3,5){\circle*{2}}
\put(-.65,5){$\vee$}
\put(6,12){\circle*{2}}
\put(0,12){\circle*{2}}
\put(3,0){\circle*{2}}
\put(3,0){\line(0,1){5}}
\put(6,-3){\tiny #1}
\put(6,4){\tiny #2}
\put(9,12){\tiny #3}
\put(-5,12){\tiny #4}
\end{picture}}

\newcommand{\tcinqonze}[5]{\begin{picture}(15,26)(-5,-1)
\put(3,5){\circle*{2}} 
\put(-0.65,5){$\vee$}
\put(6,12){\circle*{2}} 
\put(0,12){\circle*{2}}
\put(3,0){\circle*{2}} 
\put(3,0){\line(0,1){5}}
\put(0,12){\line(0,1){7}} 
\put(0,19){\circle*{2}} 

\put(5,0){\tiny #1}
\put(5,5){\tiny #2}
\put(-5,12){\tiny #3}
\put(-2,21){\tiny #4}
\put(7,11){\tiny #5}
\end{picture}}

\begin{document}

\title{Double Shuffle relations for Arborified Zeta Values.}
\author{Pierre~J.~Clavier${}^{1}$\\
\normalsize \it $^1$  Institute of Mathematics, \\
\normalsize \it University of Potsdam,\\
\normalsize \it D-14476 Potsdam, Germany\\
~\\
\normalsize email: clavier@math.uni-potsdam.de}

\date{}

\maketitle

\begin{abstract} 
 Arborified zeta values are defined as iterated series and integrals using the universal properties of rooted trees. This approach allows to study their convergence domain and to relate them to  {multiple zeta} values. Generalisations to 
 rooted trees of the stuffle and shuffle products are defined and studied. It is further shown that 
 arborifed zeta values are algebra morphisms for these new products on trees.
\end{abstract}

\tableofcontents

\section{Introduction}

\subsection{Aims of the paper} 

Arborified zeta values is a generalisation of {multiple zeta} values that have not yet been fully explored. This paper aims at closing this gap. More precisely, 
the aims of this paper are threefold:
\begin{itemize}
 \item To rigorously define arborified zeta values and prove their domain of convergence.
 \item To relate arborified zeta values to  {multiple zeta} values.
 \item To study  relations obeyed by the arborified zeta values.
\end{itemize}
The approach adopted here is inspired by \cite{CGPZ1} and \cite{MP} where the focus was on renormalised values yet the algebraic and analytic tools used in those papers can also be implemented in the convergent case 
studied here. The new input of the present work is
\begin{enumerate}
 \item The application of the analytic tools of \cite{MP} and \cite{Pa} to arborified zeta values to characterise their convergence domain.
 \item The application of the algebraic tools of \cite{CGPZ1} and \cite{CGPZ2} to study the relations obeyed by arborified zeta values.
 \item The construction of new (to the author's knowledge) algebraic structures on trees relevant to arborified zeta values.
 \item A relation between the two versions of arborified zeta values present in \cite{M}.
\end{enumerate}
The first two results above are obtained in Sections \ref{section:stuffle} (for arborified zeta values defined as iterated series) and \ref{section:shuffle} (for arborified zeta values defined as iterated integrals).
The third result is achieved in Section  \ref{section:shuffle_tree}, where the shuffles of trees are defined and their relations to arborified 
zeta values explored. The last achievement is presented in the Appendix \ref{section:further} where Hoffman's relations are 
studied in the context of arborified zeta values.

\subsection{State of the art}

 {Multiple zeta} values (MZVs)\footnote{also called ``multizeta numbers''  by \'Ecalle, ``multiple harmonic sums'' by Hoffman, ``Euler-Zagier numbers'' by the Borwein brothers and ``polyzeta 
numbers'' (in order to respect Weil's principles) by Cartier. We will follow  {the referee's suggestion and} Zagier's denomination of  {multiple zeta} values as it seems to be the most widespread nowadays.} have by now a  substantial history as they can be traced back 
to Euler \cite{Eu}. Over the following two centuries, they were nearly forgotten although 
they would  appear here and there in various independent works. In the 80s, MZVs have arisen in \'Ecalle's \cite{Ec} theory of mould 
calculus. A systematic study of MZVs was later initiated by Hoffman \cite{Ho1} and Zagier \cite{Za}.

MZVs are nowadays a well-established subject, with many known results concerning their algebraic and number-theoretic properties, and ambitious conjectures. For a more detailed account of the historical background and recent 
developments, we refer the reader to one of the many available very good introductions to MZVs, for example \cite{Wa}. Yet let us 
list known results on MZVs which we will generalise to trees.

Let $\Omega$ be a set. In the following, we write $\calw_\Omega$ for the algebra freely generated over $\R$ by words written in the alphabet $\Omega$. This algebra can also be 
seen as the algebra of noncommutative polynomials with variables in $\Omega$ over $\R$. {\bf Stuffle MZVs} can be seen as a map 
\begin{align}
 \zeta_\stuffle :~ & \calw_{\N^*}^{\rm conv}\subseteq\calw_{\N^*} \longrightarrow \R \nonumber \\
		  & (s_1\cdots s_k) \mapsto \sum_{n_1>\cdots n_k>0}^\infty\prod_{i=1}^k n_i^{-s_i} \label{eq:series_stuffle_MZV}
\end{align}
where $\calw_{\N^*}^{\rm conv}$ is a subset of $\calw_{\N^*}$ on which the series in \eqref{eq:series_stuffle_MZV} are well-defined. {\bf Shuffle MZVs} on the other hand, are given by a map 
\begin{align}
 \zeta_\shuffle :~ & \calw_{\{x,y\}}^{\rm conv}\subseteq\calw_{\{x,y\}} \longrightarrow \R \nonumber \\
		  & (\epsilon_1\cdots \epsilon_k) \mapsto \int_{1\geq t_1\geq \cdots \geq t_k\geq0}\prod_{i=1}^k \omega_{\epsilon_i}(t) \label{eq:integral_shuffle_MZV}
\end{align}
with $\omega_x(t)=dt/t$, $\omega_y(t)=dt/(1-t)$ and where $\calw_{\{x,y\}}^{\rm conv}$ is a subset of $\calw_{\{x,y\}}$ ensuring convergence of 
the integral in \eqref{eq:integral_shuffle_MZV}. More precise definitions of 
stuffle and shuffle MZVs will be given in Sections \ref{section:stuffle} and \ref{section:shuffle} respectively. The terminology comes from the simple, yet crucial, observation that 
stuffle (resp. shuffle) MZVs are algebra morphisms for the stuffle\footnote{also called ``quasi-shuffle'' and ``sticky shuffle'', shortened in ``stuffle''.} (resp. shuffle) product:
\begin{equation} \label{eq:stuffle_shuffle_prod}
 \zeta_\stuffle(w\stuffle w') = \zeta_\stuffle(w)\zeta_\stuffle(w'), \qquad \zeta_\shuffle(w\shuffle w') = \zeta_\shuffle(w)\zeta_\shuffle(w').
\end{equation}
The stuffle $\stuffle$ and shuffle $\shuffle$ products are rigorously defined in Definition \ref{defn:shuffle_prod}. 

Stuffle and shuffle MZVs are linked through what we call the {\bf binarisation map}
\begin{align} \label{eq:binarisation_map}
 \fraks :\calw_{\N^*}&\longrightarrow\calw_{\{x,y\}} \\
 (n_1\cdots n_k) & \longrightarrow (\underbrace{x\cdots x}_{n_1-1}y\cdots\underbrace{x\cdots x}_{n_k-1}y). \nonumber
\end{align}
This result is based on an observation of Kontsevitch, as reported in \cite[section 9]{Za}.
Indeed, a fundamental result of the theory of MZVs is that, for any convergent word $w\in\calw_{\N^*}$, one has
\begin{equation} \label{eq:shuffle_stuffle_words}
 \zeta_\shuffle(\fraks(w)) = \zeta_\stuffle(w).
\end{equation}
This, together with the fact that $\fraks\left(\calw_{\N^*}^{\rm conv}\right) = \calw_{\{x,y\}}^{\rm conv}$ implies that 
${\rm Im}(\zeta_\stuffle) = {\rm Im}(\zeta_\shuffle)$
and justifies the name `` {multiple zeta} \emph{values}''; i.e. that we identify these maps and the elements of their image.

The third set of relations that MZVs obey are {\bf Hoffman's regularisation relations} \cite{Ho1,Ho2}: for any convergent word $w$, 
$\fraks\left((1)\stuffle w)\right) - (y)\shuffle\fraks(w)$ is a convergent word and
\begin{equation} \label{eq:Hoffman_reg_rel}
 \fraks\left((1)\stuffle w)\right) - (y)\shuffle\fraks(w) \in \Ker(\zeta_\shuffle).
\end{equation}
The shuffle, stuffle and Hoffman's regularisation relations are merged under the common denomination {\bf regularised double shuffle relations}. 
While other relations among MZVs are known (e.g. the 
duality relation) it is conjectured that regularised double shuffle relations generate every algebraic relations between MZVs.

MZVs have many generalisations, conical zeta values \cite{GPZ}, Hurwitz MZVs \cite{Bo}, elliptic MZVs \cite{En}, Witten's MZVs \cite{Wi} 
among others. 
{\bf Arborified zeta values} (AZVs)\footnote{also called ''branched zeta values`` in \cite{CGPZ1}}  are one generalisation to trees of MZVs. 
While these objects are also present in \'Ecalle's work \cite{Ec}, they were not, by far, as extensively studied as MZVs. A systematic study of 
arborified MZVs started only of 2016 with the work of D.~Manchon 
\cite{M}. They did, however, appear two years earlier in the work of Yamamoto \cite{Y}.

In \cite{M}, the stuffle and shuffle\footnote{respectively called ''contracting arborification`` and ''simple arborification`` in \cite{M}, 
 thus following  in that the names given by 
\'Ecalle.} versions of AZVs are presented and some 
of their properties are shown. Furthermore the question of lifting the map $\fraks$ to trees is raised in that same paper.

Finally, let us mentioned that (divergent) AZVs were defined and studied in \cite{CGPZ1} as a testing ground for multivariate renormalisation 
techniques.

\subsection{Content of the paper}

Section \ref{section:generalities} introduces the concepts used in the rest of the paper. In particular, the branching procedures of 
\cite{CGPZ1} are recalled in Definitions \ref{defn:phi_hat} and \ref{defn:phi_sharp}. 
We give in Section \ref{section:generalities} a purely combinatorial proof of Theorem \ref{thm:flattening} which links Rota-Baxter 
operators and the branching procedure of Definition \ref{defn:phi_hat}. The original  
proof of this result  can be found in \cite[Theorem 2.13]{CGPZ1}.

A procedure was presented in \cite{CGPZ1} 
which takes advantage of the universal property of rooted forest (Theorem \ref{thm:univ_prop_tree}) to lift a map $\phi:A\longrightarrow A$ to a 
morphism of operated algebras $\widehat\phi:\calf_A\longrightarrow A$. $\widehat\phi$ is called the {\bf branching} of $\phi$. 
In Section \ref{section:stuffle} we apply this procedure to build arborified zeta values as iterated series.

Following \cite{MP,CGPZ1}, we use branching procedures on the Euler-MacLaurin operator on classical symbols. These 
techniques allow us to keep track of the order of symbols, see Proposition \ref{prop:order_calz}. Stuffle arborified zeta values are then defined 
(and a convergence criterion given) in 
Definition-Proposition \ref{defnprop:arborified_zeta}. Stuffle AZVs are then shown to be algebra morphisms for the concatenation product of 
trees in Proposition \ref{prop:stuffle_alg_mor}.

Later in Section \ref{section:stuffle}, the same techniques are further applied to words to built MZVs. The order of symbols obtained after the iteration of 
the Euler-MacLaurin operator is given in 
Theorem \ref{thm:calz_words_order}, and the definition of MZVs (also with a convergence criterion) in Definition-Proposition 
\ref{defnprop:zeta_stuffle}. This allows to 
prove the following first main result of this paper  (Theorem \ref{thm:main_result_stuffle}):
\begin{customthm}{1}
 For any convergent forest $F$, the convergent  arborified zeta value $\zeta^T_\stuffle(F)$ (resp. $\zeta^{T,\star}_\stuffle(F)$) is a finite 
 linear combination of convergent  {multiple zeta} values $\zeta_\stuffle(w)$ (resp. $\zeta^{\star}_\stuffle(w)$) with rational coefficients. More precisely, it can be written as a finite 
 linear combination of  {multiple zeta} values with integer coefficients.
\end{customthm}
The two statements of this result, which differ only by the words rational/integer is a consequence from the fact that one can choose with which MZVs
we write the $\zeta^T_\stuffle(F)$. We claim that there is one particular choice for which the rational coefficients are all integers. 
Another choice would lead to a linear combination of 
MZVs with rational coefficients rather than integer (under the standard conjecture that the shuffle, stuffle and Hoffman's regularisation relations are the only 
rational relations between MZVs). This observation will also hold for the next results relating arborified zeta and arborified polylogarithms to their non-arborified counterparts.

The techniques of Section \ref{section:generalities} are used in Section \ref{section:shuffle} to build arborified zeta values as iterated 
integrals in a different fashion from that of  Section \ref{section:stuffle}. 
Indeed, in Section \ref{section:stuffle}, branching procedures were used in a given ambient space (the algebra of classical symbols). In 
Section \ref{section:stuffle}, we relate the arborified versions of iterated
integrals to their unbranched counterparts at each steps. This is done for practical purpose, yet it illustrates the flexibility of branching 
procedures.

Since shuffle MZVs are defined by Chen's iterated integral, we
recall the definition of Chen's iterated integrals in Definition-Proposition \ref{defnprop:Chen_int}. This allows us 
to define arborified Chen integrals in Definition \ref{defn:arborified_chen_int}. The relation between these two objects is given in 
Proposition \ref{prop:chen_int_arbo_words}. In a special case, arborified Chen integrals 
are arborified polylogarithms, whose definition (together with a convergence criterion) is given in Definition-Proposition 
\ref{defnprop:arbo_polylogs}. Similar techniques are used to defined the usual 
multiple polylogarithms in Definition \ref{defn:multiple_polylogs}. A second important result of this paper is then Theorem 
\ref{thm:arborified_polylogs} which states
\begin{customthm}{2}
 For any semiconvergent forest $F$, the arborified polylogarithm associated to $F$ enjoys the following properties
 \begin{enumerate}
  \item it is a $\Q$-linear sum of multiple polylogarithms;
  \item it is a smooth map on $[0,1[$;
  \item The arborified polylogarithm map $Li^T:F\mapsto Li^T_F$ is an algebra morphism for the concatenation of trees and the 
 pointwise product of functions.
 \end{enumerate}
\end{customthm}
Shuffle AZVs (resp. MZVs) are then introduced in Definition \ref{defn:shuffle_AZVs} (resp \ref{defn:shuffle_MZVs}). These two objects are 
related in Theorem \ref{thm:main_result_shuffle}, which reads
\begin{customthm}{3} 
 For any convergent forest $F\in\calf_{\{x,y\}}^{\rm conv}$, $\zeta_\shuffle^T(F)$ is a $\Q$-linear combination of  {multiple zeta} values that can be written as a finite 
 linear combination of  {multiple zeta} values with integer coefficients. 
 Furthermore the map $\zeta_\shuffle^T:\calf_{\{x,y\}}^{\rm conv} \longrightarrow \R$
 is an algebra morphism for the concatenation product of trees.
\end{customthm}
Let us emphasize that the techniques used to build arborified objects also allow to prove results about the usual, unbranched objects; for example MZVs and polylogarithms. Beside the 
convergence criteria already mentioned, we have been able 
to provide new (to the author's knowledge) proofs that stuffle zeta are stuffle morphism (see Proposition \ref{prop:zeta_stuffle_mor}) 
and new proofs that multiple polylogarithms
and shuffle MZVs are algebra morphisms for the shuffle product in Propositions \ref{prop:polylogs_shuffle_mor} and 
\ref{prop:shuffle_zeta_alg_mor_shuffle} respectively.

In Section \ref{section:shuffle_tree} we define (Definition \ref{defn:shuffle_tree}) $\lambda$-shuffle products on trees. We show in Proposition
\ref{prop:shuffles_trees} that these products equip rooted forests with nonassociative algebra structures. Theorem 
\ref{thm:branching_shuffle_tree_morphism} states that branchings of Rota-Baxter operators are algebra morphisms for these $\lambda$-shuffle products 
on trees. This result is applied in Theorem \ref{thm:AZV_alg_morphism_shuffle} to show that the various versions of the AZVs are 
algebra morphisms for specific $\lambda$-shuffle products on trees:
\begin{customthm}{4} 
 The map $\zeta^T_\stuffle:\calf_{\N^*}^{\rm conv}\longrightarrow\R$ (resp. $\zeta^{T,\star}_\stuffle:\calf_{\N^*}^{\rm conv}\longrightarrow\R$, resp. $\zeta^T_\shuffle:\calf_{\{x,y\}}^{\rm conv}\longrightarrow\R$) is an 
 algebra morphism for the stuffle (resp. anti-stuffle, resp. shuffle) product on trees.
\end{customthm}
This induces relations amongst AZVs that have no direct equivalent for MZVs (Corollary \ref{coro:relation_nonasso}).

We finish this article with an Appendix, whose purpose is the arborification of other properties of MZVs, namely the binarisation map 
\eqref{eq:binarisation_map} and Hoffman's regularisation relations \eqref{eq:Hoffman_reg_rel}. 
The branched version of the binarisation map is given in Definition \ref{defn:branched_bin_map}. The main result of this Appendix is 
Theorem \ref{thm:relation_shuffle_stuffle}:
\begin{customthm}{5}
 For any convergent forest $F\in\calf_{\N^*}^{\rm conv}$ we have
 \begin{equation*}
  \zeta^T_{\shuffle}(\fraks^T(F)) \leq \zeta^T_{\stuffle}(F).
 \end{equation*}
 Furthermore, the inequality is an equality if and only if $F$ has no branching vertex (i.e. $F$ is the empty tree or $F=l_1\cdots l_k$ with 
 $l_i$ being ladder trees).
\end{customthm}
Finally, Propositions \ref{prop:branched_Hoffman} and \ref{prop:conv_Hoffman_trees} give two arborified versions of Hoffman's 
regularisation relations.

While completing the tasks set for this paper, new structures have been unraveled (in particular in Section \ref{section:shuffle_tree}). It leads to further natural questions, such as 
their links to known structures on trees, in particular Connes-Kreimer's coproduct. Moreover it is pointed out in the Appendix that Hoffman's regularisation relations do not naturally lift 
to trees. We feel that the quantities arising in this discrepancy are worth studying. Another possible approach could be to study another generalisation $\zeta^t$ of MZVs to trees, for which the relation 
$\zeta^t_{\shuffle}(\fraks^T(F)) = \zeta^t_{\stuffle}(F)$ holds.

\section{Trees, words and universal properties} \label{section:generalities}

Unless otherwise specified, the word algebra will stand in this paper for a unital, associative algebra over $\R$. We also set 
$\N:=\Z_{\geq0}$ and $\N^*:=\Z_{\geq1}$.

\subsection{Generalities on trees and words}

We recall here some well-known definitions and results on trees, see for example \cite{Fo}.
\begin{defn}
 A {\bf tree} $T$ is a finite connected loopless graph: $T=(V(T);E(T))$. We use the short-hand notation $\emptyset$ for the {\bf empty graph} 
 $(\emptyset,\emptyset)$. A 
 {\bf rooted tree} is a tree together with a partial order $\geq$ on the set of vertices such that this partial order has a minimum element, 
 called 
 {\bf the root}. A {\bf rooted forest} is a finite disjoint union of rooted trees. The partial orders on the vertices of each tree induce a 
 partial order on the vertices of the forest. We write $\calf$ the commutative algebra freely generated by rooted forests.

 A vertex that is maximal for the partial order on the vertices is called a {\bf leaf}. If for two vertices $v$ and $v'$ are such that 
 $v'\geq v$, 
 $v'\neq v$ and for any vertex $v''$, $v'\geq v''\geq v$ implies $v'=v''$ or $v=v''$, then $v'$ is called a {\bf direct successor} to $v$. A 
 vertex that has strictly more than one direct successor is called a {\bf branching vertex}. A tree with non branching vertex is a 
 {\bf ladder tree}.

 Let $\Omega$ be a set. A {\bf $\Omega$-decorated rooted forest} is a rooted forest $F$ together with a {\bf decoration map} 
 $d:V(F)\mapsto\Omega$. When there is no need to specify the decoration map we simply write $F$ for a decorated forest $(F,d)$. Let 
 $\calf_\Omega$ be the commutative algebra freely generated by $\Omega$-decorated rooted  {trees}.
\end{defn}
We now fix our notation for words.
\begin{defn}
 For a set $\Omega$, we write $\calw_\Omega$ the linear span (over $\R$) of {\bf words} written in the alphabet $\Omega$. $\calw_\Omega$ is 
 therefore the algebra over $\R$ of non-commutative polynomials with variables in $\Omega$. We also write $\emptyset$ for the empty word. 
 Furthermore we write 
 $\iota_\Omega:\calw_\Omega\hookrightarrow\calf_\Omega$ the canonical injection which sends the empty word to the empty tree and non empty 
 words to ladder trees:
 \begin{equation*}
  \iota_\Omega(\omega_1\cdots\omega_k):=B_+^{\omega_1}\circ\cdots\circ B_+^{\omega_k}(\emptyset)
 \end{equation*}
 with $B_+$ the branching operator on forests, to be defined below.
\end{defn}
Finally, let us define some gradings on trees and words.
\begin{defn}
 \begin{itemize}
  \item Let $w$ be a word written in the alphabet $\Omega$. We define its {\bf length} $|w|$ to be $0$ if $w$ is the empty tree and 
  $|\omega_1\cdots\omega_k|:=k$ otherwise.
  \item Similarly, let $F=(V(F),E(F))$ be a rooted forest. We set $|F|:=|V(F)|$.
  \item For any $\omega\in\Omega$ and $F$ a forest decorated by $\Omega$, let $\sharp_\omega F$ the number of vertices of $F$ decorated by 
  $\omega$:
  \begin{equation*}
   \sharp_\omega F:= |\{v\in V(F):d(v)=\omega\}|.
  \end{equation*}
  \item Similarly, we write $\sharp_\omega w$ the number of times that the letter $\omega$ appears in a word $w$.
  \item Let $(\Omega,\bullet)$ be a commutative semigroup. We define the {\bf weight with respect to the product $\bullet$} $||w||_\bullet$ of a word $w\in\calw_\Omega$ to be $0$ is $w=\emptyset$ and 
  $||\omega_1\cdots\omega_k||_\bullet:=\omega_1\bullet\cdots\bullet\omega_k$. If the product on $\Omega$ is clear from context, we will speak of the weight of $w$ and write $||w||$.
  \item Similarly, let $(\Omega,\bullet)$ be a commutative semigroup and $(F,d)=((V(F),E(F)),d)\in\calf_\Omega$ be a rooted forest. We set 
  $||F||_\bullet:=\sum_{v\in V(F)}^\bullet d(v)$; where the sum is for the product $\bullet$.
 \end{itemize}

\end{defn}

\subsection{Branching procedures}

We start by recalling the definition of operated structures \cite{G1}, as written in \cite{CGPZ2}. 
\begin{defn}
 Let $\Omega$ be a set. An {\bf $\Omega$-operated set} (resp. {\bf semigroup, monoid, vector space, algebra}) is a set 
 (resp. semigroup, monoid, vector space, algebra) $U$ together with a map $\beta:\Omega\times U\mapsto U$.
\end{defn}
Let $\Omega$ be a set and $(U,\beta)$ be an $\Omega$-operated set (resp. semigroup, monoid, vector space, algebra). For any $\omega$ in 
$\Omega$ we write $\beta^\omega: U\mapsto U$ the map defined by $\beta^\omega(u):=\beta(\omega,u)$.  Notice that we do not require these maps 
$\beta^\omega$ to fulfill any compatibility conditions with the semigroup (resp. monoid, vector space, algebra) structure of the operated set $U$.
\begin{ex}
 Let $\Omega$ be a set and $B_+:\Omega\times\calf_\Omega\mapsto\calf_\Omega$ be the operation defined through the grafting operator which, to 
 a doublet $(\omega,F=T_1\cdots T_k)$ associates the decorated tree obtained from $F$ by adding a root decorated by $\omega$ linked to each 
 root of $T_i$ for $i$ going from $1$ to $k$ as shown for some small trees\footnote{ {the code to generate these trees was written by Lo\"ic Foissy.}} below:
 \begin{equation*}
  B_+^\omega(\emptyset) = \tdun{$\omega$} \qquad B_+^\omega(\tdun{$\omega'$}) = \tddeux{$\omega$}{$\omega'$} \qquad 
  B_+^\omega(\tdun{$\omega'$}\tdun{$\omega''$}~) = \tdtroisun{$\omega$}{$\omega''$}{$\omega'$}.
 \end{equation*}
 Then $(\calf_\Omega,B_+)$ is an $\Omega$-operated (commutative) algebra.
\end{ex}
This example enjoys a universal property as will be recalled below. This universal property was originally shown in \cite{KP}, and formulated 
in the present form in \cite{G1} and an alternative proof of this result can be found in \cite{CGPZ2}. In order to state this property and derive some of its consequences we need
to define the notion of morphism between operated structures.
\begin{defn}
 Let $\Omega$ be a set, $(U,\beta_U)$ and $(V,\beta_V)$ be two $\Omega$-operated sets (resp. semigroups, monoids, vector spaces, algebras). A 
 {\bf morphism of $\Omega$-operated sets} (resp. semigroups, monoids, vector spaces, algebras) between $U$ and $V$ is a map (resp. a semigroup 
 morphism, a monoid morphism, a linear map, an algebra morphism) $\phi:U\mapsto V$ such that, for any $\omega$ in $\Omega$ and $u$ in $U$
 \begin{equation*}
  \phi(\beta_U(\omega,u)) = \beta_V(\omega,\phi(u)).
 \end{equation*}
 In other words, $\phi$ is such that diagram \ref{fig:operated_morphism} commutes.
\end{defn}
 \begin{figure}[h!] 
  		\begin{center}
  			\begin{tikzpicture}[->,>=stealth',shorten >=1pt,auto,node distance=3cm,thick]
  			\tikzstyle{arrow}=[->]
  			
  			\node (1) {$\Omega\times U$};
  			\node (2) [right of=1] {$U$};
  			\node (3) [below of=1] {$\Omega\times V$};
  			\node (4) [right of=3] {$V$};

  			\path
  			(1) edge node [above] {$\beta_U$} (2)
  			(1) edge node [left] {$ {I_\Omega\times}\phi$} (3)
  			(3) edge node [below] {$\beta_V$} (4)
  			(2) edge node [right] {$\phi$} (4);
  			
  			\end{tikzpicture}
  			\caption{morphism of operated structures. }\label{fig:operated_morphism}
  		\end{center}
  	\end{figure}

\begin{thm} \cite{KP,G1} \label{thm:univ_prop_tree}
 Let $\Omega$ be a set. Then $\calf_\Omega$ is the initial object in the category of commutative $\Omega$-operated algebras, i.e. for any 
 commutative $\Omega$-operated algebra $(A,\beta)$,  {there exists} a unique algebra morphism 
 $\Phi:\calf_\Omega\mapsto A$ such that diagram \ref{fig:univ_prop_forests} commutes
 \begin{figure}[h!] 
  		\begin{center}
  			\begin{tikzpicture}[->,>=stealth',shorten >=1pt,auto,node distance=3cm,thick]
  			\tikzstyle{arrow}=[->]
  			
  			\node (1) {$\calf_\Omega$};
  			\node (2) [right of=1] {$\calf_\Omega$};
  			\node (3) [below of=1] {$A$};
  			\node (4) [right of=3] {$A$};

  			\path
  			(1) edge node [above] {$B_+^\omega$} (2)
  			(1) edge node [left] {$\Phi$} (3)
  			(3) edge node [below] {$\beta^\omega$} (4)
  			(2) edge node [right] {$\Phi$} (4);
  			
  			\end{tikzpicture}
  			\caption{Universal property of forests.}\label{fig:univ_prop_forests}
  		\end{center}
  	\end{figure}
 for every $\omega$ in $\Omega$.
\end{thm}

These universal properties allow to lift maps on the decorating sets to maps on forests; using the original map to define an operation. This can be carried out in various ways, and we introduce 
here those that will be used later on. Such branchings were  introduced in \cite{CGPZ1} and further used in \cite{CGPZ2}.
\begin{defn} \label{defn:phi_hat}
 Let $\Omega$ be a commutative algebra and $\phi:\Omega\mapsto\Omega$ be a map. Let $\beta_\phi:\Omega\times\Omega\mapsto\Omega$ be the 
 operation of $\Omega$ on itself defined by 
 $\beta_\phi(\omega_1,\omega_2) := \phi(\omega_1.\omega_2)$. The {\bf branched $\phi$-map} (or branching of $\phi$) is the morphism of 
 commutative $\Omega$-operated algebras 
 $\widehat{\phi}:(\calf_\Omega,B_+)\mapsto(\Omega,\beta_\phi)$ whose existence and unicity is given by Theorem \ref{thm:univ_prop_tree}.
\end{defn}
Notice that the map $\widehat\phi$ is entirely determined by the relations
\begin{align}
 \widehat\phi(\emptyset) & = 1_\Omega, \nonumber \\
 \widehat\phi(F_1F_2) & = \widehat\phi(F_1)\widehat\phi(F_2), \label{eq:prod_hat_phi} \\
 \widehat\phi(B_+^\omega(F)) & = \phi\left(\omega\widehat\phi(F)\right). \nonumber
\end{align}

\begin{ex}
 Let $\Omega$ and $\phi$ be as in the above Definition. We give some examples of the action of $\widehat\phi$:
 \begin{align*}
  \widehat{\phi}(\emptyset) = 1_\Omega; \qquad & \widehat{\phi}(\tdun{$\omega$}) = \phi(\omega); \qquad \widehat{\phi}\left(\tddeux{$\omega_1$}{$\omega_2$} \right) = \phi\left(\omega_1\phi(\omega_2)\right); \\
  \widehat{\phi}\left(\tdtroisun{$\omega_1$}{$\omega_2$}{~$\omega_3$}\right) & =  \phi\left(\omega_1\phi(\omega_2)\phi(\omega_3)\right).
 \end{align*}
\end{ex}

\begin{defn} \label{defn:phi_sharp}
 Let $\Omega_1,\Omega_2$ be two commutative algebras and $\phi:\Omega_1\mapsto\Omega_2$ be a map. Let 
 $\tilde\beta_\phi:\Omega_1\times\calf_{\Omega_2}\mapsto\calf_{\Omega_2}$ be the operation of $\Omega_1$ on 
 $\calf_{\Omega_2}$ defined by  $\tilde \beta_{\phi}(\omega_1,F) := B_+(\phi(\omega_1),F)$. The {\bf lifted $\phi$-map} (or lifting of $\phi$) 
 is the morphism of commutative $\Omega_1$-operated algebras 
 $\phi^\sharp:(\calf_{\Omega_1},B_+)\mapsto(\calf_{\Omega_2},\tilde\beta_{\phi})$ whose existence and unicity is given by Theorem 
 \ref{thm:univ_prop_tree}.
\end{defn}

\begin{ex}
 Let $\Omega_1,\Omega_2$ and $\phi$ be as in the above Definition. We then have
 \begin{align*}
  \phi^\sharp(\emptyset) = \emptyset; \qquad & \phi^\sharp(\tdun{$\omega$}) = \tdun{$\phi(\omega)$}\hspace{0.4cm}; \qquad \phi^\sharp\left(\tddeux{$\omega_1$}{$\omega_2$} \right) = \tddeux{$\phi(\omega_1)$}{$\phi(\omega_2)$}\hspace{0.5cm}; \\ 
  \phi^\sharp\left(\tdtroisun{$\omega_1$}{$\omega_2$}{~$\omega_3$}\right) & =  \tdtroisun{$\phi(\omega_1)$}{$\phi(\omega_2)$}{$\phi(\omega_3)\quad$}\hspace{0.6cm}.
 \end{align*}

\end{ex}

\subsection{From trees to words}

  We recall here some well-known structures that one can build on the set of words written in a given alphabet, referring the reader to one of 
  the many introductions of the combinatorics of words (e.g. \cite{G2}) for a more detailed exposition. 
\begin{defn} \label{defn:shuffle_prod}
 \begin{enumerate}
  \item The {\bf concatenation product} $\sqcup:\calw_\Omega\times\calw_\Omega\mapsto\calw_\Omega$ is defined by
  \begin{align*}
   \emptyset\sqcup w = w\sqcup\emptyset & = w, \\
   (\omega_1\cdots\omega_k)\sqcup(\omega_1'\cdots\omega_n') &  = (\omega_1\cdots\omega_k\omega_1'\cdots\omega_n')
  \end{align*}
  for any word $w$ in $\calw_\Omega$ and letters $\omega_1,\cdots,\omega_k,\omega_1',\cdots,\omega_n'$ in $\Omega$.
  \item Let $\Omega$ (resp. $(\Omega,\bullet)$) be a set (resp. a commutative semigroup). The {\bf shuffle product} (resp. the 
  {\bf $\lambda$-shuffle product}) is recursively defined by
  \begin{align*}
   \emptyset\shuffle w = w\shuffle\emptyset & = w, \\
   \left((\omega)\sqcup w\right) \shuffle \left((\omega')\sqcup w'\right) & = (\omega)\sqcup\left[w \shuffle \left((\omega')\sqcup w'\right) \right] + (\omega')\sqcup\left[\left((\omega)\sqcup w\right)\shuffle w'\right] 
  \end{align*}
  (resp. by
  \begin{align*}
   \emptyset\shuffle_\lambda w = w\shuffle_\lambda\emptyset & = w, \\
   \left((\omega)\sqcup w\right) \shuffle_\lambda \left((\omega')\sqcup w'\right) & = (\omega)\sqcup\left[w \shuffle_\lambda \left((\omega')\sqcup w'\right) \right] + (\omega')\sqcup\left[\left((\omega)\sqcup w\right)\shuffle_\lambda w'\right]  \\
										  & + \lambda(\omega\bullet\omega')\sqcup\left[w\shuffle_\lambda w'\right] 
  \end{align*}
  for any $\lambda$ in $\R$) extended by bilinearity to products 
  $\shuffle,\shuffle_\lambda:\calw_\Omega\times\calw_\Omega\longrightarrow\calw_\Omega$.
 \end{enumerate}

\end{defn}

\begin{rk} \label{rk:set_monoid}
 \begin{itemize}
  \item We notice that $\shuffle=\shuffle_0$. Therefore we will not distinguish the shuffle and $\lambda$-shuffle and always 
  implies  that the stated results will hold 
  for $\Omega$ a set without a semigroup structure when $\lambda =0$.
  \item It is well-known that $\shuffle_\lambda$ is associative and commutative for any $\lambda$ in $\R$ (see e.g. \cite{Ho3}).
  \item For $\lambda=1$, one finds back the {\bf stuffle product} (or sticky shuffle) written $\stuffle$.
  \item For $\lambda=-1$, we call the product $\shuffle_{-1}$ the {\bf anti-stuffle} product.
 \end{itemize}

\end{rk}
Following \cite{CGPZ2}, we can now define a map from forests to words. It follows from the observation that 
$(\calw_\Omega,\emptyset,\shuffle_\lambda)$ is a commutative algebra operated by $\Omega$ through the natural map 
\begin{align} \label{eq:defn_c_plus}
 C_+ : ~ \Omega & \times\calw_\Omega \mapsto\calw_\Omega \\
         (\omega & ,w) \mapsto C_+^\omega(w):=(\omega)\sqcup w. \nonumber
\end{align}
\begin{defn} \label{defn:flattening}
 Let $\Omega$ be a commutative semigroup and $\lambda$ a real number. The {\bf flattening map of weight $\lambda$} is the morphism of 
 $\Omega$-operated algebras $fl_\lambda:(\calf_\Omega,B_+)\mapsto(\calw_\Omega,C_+)$ whose 
 existence and unicity is given by Theorem \ref{thm:univ_prop_tree}. 
\end{defn}
Notice that this map is entirely determined by the following relations:
\begin{align*}
 fl_\lambda(\emptyset) & = \emptyset \\
 fl_\lambda(F_1F_2) & = fl_\lambda(F_1)\shuffle_\lambda fl_\lambda(F_2) \\
 fl_\lambda(B_+^\omega(F)) & = (\omega)\sqcup fl_\lambda(F).
\end{align*}
Before we move on, let us state a simple yet important property of this  
flattening map.
\begin{prop} \label{prop:finite_Q_sum_flattening}
 Let $\Omega$ be a commutative semigroup and $\lambda$ a rational (resp. integer) number. For any finite forest $F$ in 
 $\calf_\Omega$, $fl_\lambda(F)$ is a 
 finite sum of finite words with rational (resp. integer) coefficients.
\end{prop}
\begin{proof}
 This result is easily shown by induction on the number of vertices of $F$. It follows from the fact that the flattening of the empty 
 forest is trivially a finite sum of finite words with integer coefficients; that (provided $\lambda$ is rational; resp. integer) the 
 $\lambda$-shuffle 
 of two finite sums of finite words with rational (resp. integer) coefficients is a finite sum of finite words with rational (resp. integer) 
 coefficients since 
 $\Q$ (resp. $\Z$) is stable under multiplication and addition; and 
 finally that the concatenation by one letter of a finite sum of finite words with rational (resp. integer) coefficients is a finite sum of 
 finite words 
 with rational (resp. integer) coefficients.
\end{proof}

\subsection{Flattening and Rota-Baxter maps}

As in \cite{CGPZ2}, one builds counterparts of Definitions \ref{defn:phi_hat} and \ref{defn:phi_sharp} in the context of words.
\begin{defn}
 Let $\Omega$ be a commutative algebra, $\Omega_1$ and $\Omega_2$ be two sets. Let $\iota_\Omega$ 
 (resp. $\iota_{\Omega_1}$, $\iota_{\Omega_2}$) be the canonical inclusion of words into trees, which to any word associate a 
 ladder tree. Let $\phi:\Omega\mapsto\Omega$ and $\psi:\Omega_1\mapsto\Omega_2$ be two maps. The {\bf W-branched $\phi$ map} 
 $\widehat\phi_\calw:\calw_\Omega\mapsto\Omega$ is defined by 
 $\widehat\phi_\calw:=\widehat\phi\circ\iota_\Omega$; and the {\bf W-lifted $\phi$ map} $\phi^\sharp_\calw:\calw_\Omega\mapsto$ by 
 $\phi^\sharp_\calw:=\iota_{\Omega_2}^{-1}\circ\phi^\sharp\circ\iota_{\Omega_1}$.
\end{defn}
\begin{rk}
 \begin{itemize}
  \item The W-lifted $\phi$ map is well defined since $\phi^\sharp$ maps ladder trees to ladder trees, thus the image of 
  $\phi^\sharp\circ\iota_{\Omega_1}$ is included in the image of $\iota_{\Omega_2}$.
  \item The maps $\widehat\phi_\calw$ and $\phi^\sharp_\calw$ can be recursively defined by the relations:
  \begin{align*}
  \widehat\phi_\calw(\emptyset) & := 1_\Omega \\
  \widehat\phi_\calw\left((\omega)\sqcup w\right) & := \phi\left(\omega\widehat\phi_\calw(w)\right)
 \end{align*}
 for $\widehat\phi_\calw$ and 
 \begin{align*}
  \phi^\sharp_\calw(\emptyset) & := \emptyset \\
  \phi^\sharp_\calw\left((\omega)\sqcup w\right) & := \phi(\omega)\sqcup\phi^\sharp_\calw(w)
 \end{align*}
 for $\phi^\sharp_\calw$.
 \end{itemize}

\end{rk}
The various structures described above are linked through the notion of Rota-Baxter map. A gentle introduction to this rich subject can be 
found in \cite{G2}.
\begin{defn}
 Let $\Omega$ be a commutative algebra and $\lambda$ a real number. A map $P:\Omega\mapsto\Omega$ is a {\bf Rota-Baxter map of weight 
 $\lambda$} if
 \begin{equation} \label{eq:RBO}
  P(a)P(b) = P(aP(b)) + P(P(a)b) + \lambda P(ab)
 \end{equation}
 for any $a$ and $b$ in $\Omega$.
\end{defn}
Let us give standard examples of Rota-Baxter maps, which will be of importance in the next Sections.
\begin{ex} \label{ex:RBmap}
 \begin{enumerate}
  \item Let $X\subseteq L^1_{\rm loc}(\R)$ be an algebra of locally integrable functions on $\R$ equipped with the  {pointwise} product of 
  functions stable under the integration map $f\mapsto\left(x\to\int_a^xf(x)dx\right)$. For any 
  $a,b$ in $\R$, the integration map is a Rota-Baxter operator of weight $0$. \label{ex:RB_integration}
  \item Let $l(\R)$ be the algebra of sequences on $\R$ equipped with the  {pointwise} product of sequences. The summation operator 
  $\Sigma:l(\R)\mapsto l(\R)$ defined by $\Sigma(u_n)(N):=\sum_{n=0}^N u_n$ is a Rota-Baxter map of weight $-1$. \label{ex:RB_sum_star}
  \item Let $\frakt_{-1}:\N^*\mapsto\N$ be the translation map by $-1$ and $\frakt^*_{-1}$ be its pull-back to $l(\R)$. Then $\frakt^*\Sigma$ is 
  a Rota-Baxter map of weight $+1$. \label{ex:RB_sum}
 \end{enumerate}

\end{ex}

We state a result of \cite{CGPZ1} for which we provide an original, straightforward proof.
\begin{thm} \cite[Theorem 2.13]{CGPZ1} \label{thm:flattening}
 Let $\Omega$ be a commutative algebra and $P:\Omega\mapsto\Omega$ a linear map. For any real number $\lambda$, the following statements are equivalent:
 \begin{enumerate}
  \item $P$ is a Rota-Baxter map of weight $\lambda$.
  \item The $P$-branched map factorises through words: $\widehat{P}=\widehat{P}_\calw\circ fl_\lambda$. \label{thm:ii}
  \item $\widehat P_\calw$ is a morphism for the $\lambda$-shuffle product, namely 
  $\widehat P_\calw(w\shuffle_\lambda w')=\widehat P_\calw(w)\widehat P_\calw(w')$ for any $w,w'$ in $\calw_\Omega$. \label{thm:iii}
 \end{enumerate}
\end{thm}
\begin{rk}
 \begin{itemize}
  \item This statement was proven in \cite{CGPZ2} in the more general framework of locality structures. This Theorem can be seen as the case 
  where the independence relation is complete: $\top=\Omega\times\Omega$.
  \item The proof of this result was obtained in \cite{CGPZ2} by using a locality version of universal properties of words. We propose here a 
  more straightforward proof, which also holds in the locality framework.
 \end{itemize}
\end{rk}
\begin{proof}
 $(ii)\Rightarrow(i)$: For any $a$, $b$ in $\Omega$, applying $\widehat P$ to the forest $\tdun{a}\tdun{b}$ gives 
 $\widehat{P}(\tdun{a}\tdun{b})=P(a)P(b)$. On the other hand, if 
 $\widehat{P}=\widehat{P}_\calw\circ fl_\lambda$ holds one has
 \begin{align*}
  \widehat{P}(\tdun{a}\tdun{b}) & = \widehat{P}_\calw(fl_\lambda(\tdun{a}\tdun{b})) = \widehat{P}_\calw((a)\shuffle_\lambda(b)) = \widehat{P}_\calw((a)\sqcup(b)) + \widehat{P}_\calw((b)\sqcup(a)) + \lambda\widehat{P}_\calw((ab)) \\
				  & = P(aP(b)) + P(P(a)b) + \lambda P(ab);  
 \end{align*}
 thus $P$ is a Rota-Baxter map of weight $\lambda$.
 
 $(i)\Rightarrow(iii)$: Assuming that $(i)$ holds, we prove this result by induction on $p=|w|+|w'|$. Let $P$ be a Rota-Baxter operator of 
 weight $\lambda$.
 \begin{itemize}
  \item If $p=0$, then $w=w'=\emptyset$ and the result trivially holds as both sides are equal to $1$.
  \item For $p$ a natural number, let us assume that 
  $\widehat P_\calw(w\shuffle_\lambda w')=\widehat P_\calw(w)\widehat P_\calw(w')$ for any words $w$ and $w'$ such that $|w|+|w'|\leq p$. Let 
  $w$ and $w'$ be two words such that $|w|+|w'|=p+1$. If $w=\emptyset$ or $w'=\emptyset$ the result once again trivially holds. Otherwise 
  there
  exist two words $\tilde w$ and $\tilde w'$ (eventually empty) and $\omega,\omega'$ in $\Omega$ such that $w=(\omega)\sqcup\tilde w$ and 
  $w'=(\omega')\sqcup\tilde w'$. Then by definition of $\widehat P_\calw$ we have
  \begin{align*}
   \widehat P_\calw(w)\widehat P_\calw(w') & = P\left(\omega\wPw(\tilde w)\right)P\left(\omega'\wPw(\tilde w')\right) \\
					   & = P\left(\omega\wPw(\tilde w)P\left(\omega'\wPw(\tilde w')\right)\right) + P\left(P\left(\omega\wPw(\tilde w)\right)\omega'\wPw(\tilde w')\right) + \lambda P\left(\omega\wPw(\tilde w)\omega'\wPw(\tilde w')\right) \\
					   & = P\left(\omega\wPw(\tilde w)\wPw(w')\right) + P\left(\wPw(w)\omega'\wPw(\tilde w')\right) + \lambda P\left(\omega\wPw(\tilde w)\omega'\wPw(\tilde w')\right)   
  \end{align*}
  were we have used that $P$ is a Rota-Baxter map of weight $\lambda$.
  
  On the other hand we have by definition of $\shuffle_\lambda$ and from the definition of $\wPw$:
  \begin{align*}
   \wPw(w\shuffle_\lambda w') & = \wPw\left((\omega)\sqcup(\tilde w\shuffle_\lambda w')\right) + \wPw\left((\omega')\sqcup(w\shuffle_\lambda\tilde w')\right) + \lambda \wPw\left((\omega\omega')\sqcup(\tilde w\shuffle_\lambda\tilde w')\right) \\
			      & = P\left(\omega\wPw(\tilde w\shuffle_\lambda w')\right) + P\left(\omega'\wPw(w\shuffle_\lambda\tilde w')\right) + \lambda P\left(\omega\omega'\wPw(\tilde w\shuffle_\lambda\tilde w')\right) \\
			      & = P\left(\omega\wPw(\tilde w)\wPw(w')\right) + P\left(\omega'\wPw(w)\wPw(\tilde w')\right) + \lambda P\left(\omega\omega'\wPw(\tilde w)\wPw(\tilde w')\right) \\
			      & \hspace{8cm}\quad \text{by the induction hypothesis}\\
			      & = P\left(\omega\wPw(\tilde w)\wPw(w')\right) + P\left(\wPw(w)\omega'\wPw(\tilde w')\right) + \lambda P\left(\omega\wPw(\tilde w)\omega'\wPw(\tilde w')\right) 
  \end{align*}
  by commutativity of $\Omega$. This concludes this part of the proof.
 \end{itemize}
 
 $(iii)\Rightarrow(ii)$: Assuming that $(iii)$ holds, we once again prove this result by induction, this time on $p=|F|$.
 \begin{itemize}
  \item If $p=0$ then $F=\emptyset$ and the result trivially holds.
  \item For any $p$ a natural number, let us assume that $\wP(F)=\wPw(fl_\lambda(F))$ for any forest $F$ such that $|F|\leq p$. Let $F$ be a 
  forest 
  with exactly $p+1$ vertices. Therefore $F$ is nonempty and we have either $F=T=B_+^\omega(\tilde F)$ for some (eventually empty) forest 
  $\tilde F$ or $F=F_1F_2$ with $F_1$ and $F_2$ nonempty forests. In the first case we have
  \begin{align*}
   \wP(F=T=B_+^\omega(\tilde F)) & = P\left(\omega\wP(\tilde F)\right) \quad \text{by definition of }\wP \\
				 & = P\left(\omega\wPw(fl_\lambda(\tilde F))\right) \quad \text{by the induction hypothesis} \\
				 & = \wPw\left((\omega)\sqcup fl_\lambda(\tilde F)\right) \quad \text{by definition of }\wPw \\
				 & = \wPw\left(fl_\lambda(F)\right) \quad \text{by definition of }fl_\lambda.
  \end{align*}
  In the second case we have
  \begin{align*}
   \wP(F=F_1F_2) & = \wP(F_1)\wP(F_2) \quad \text{by definition of }\wP \\
		 & = \wPw\left(fl_\lambda(F_1)\right)\wPw\left(fl_\lambda(F_2)\right) \quad \text{by the induction hypothesis} \\
		 & = \wPw\left(fl_\lambda(F_1)\shuffle_\lambda fl_\lambda(F_2)\right) \quad \text{since }(iii)\text{ holds by assumption} \\
		 & = \wPw\left(fl_\lambda(F_1F_2)\right) \quad \text{by definition of }fl_\lambda.
  \end{align*}
  This conclude this induction.
 \end{itemize}
 Hence we have proven three implications, thus the theorem.
\end{proof}

\section{Stuffle arborified zeta values} \label{section:stuffle}

\subsection{Log-polyhomogeneous symbols}

We follow the presentation of \cite{MP} and \cite{Pa} to introduce log-polyhomogeneous symbols.

\begin{defn}
 \begin{enumerate}
  \item Let $r\in\R$. A smooth function $\sigma:\R_{\geq 1}\to\C$ is a {\bf symbol} (with constant coefficients) of order $r\in\R$ on 
  $\R_{\geq 1}$ if
  \begin{equation}\label{eq:symbol}
   \forall k\in\N,~\exists C_k:~\forall x\in\R_{\geq 1}, \ |\partial_x^k\sigma(x)|\leq C_k x^{r-k}.
  \end{equation}
  The set of such symbols of order $r$ on $\R_{\geq 1}$ (with constant coefficients) is written $\mathcal{S}^r(\R_{\geq 1})$, which is a real 
  vector space since $r\leq r'~\Longrightarrow S^r(\R_{\geq1})\subseteq S^{r'}(\R_{\geq1})$. 
  \item Let $\alpha\in\R$. A symbol $\sigma:\R_{\geq 1}\longmapsto\R$ is a {\bf classical symbol} also called {\bf poly-homogeneous symbol} on 
  $\R_{\geq 1}$ of order $\alpha\in\R$ if it  lies in $\mathcal{S}^{\alpha}(\R_{\geq 1})$ and if there exists a sequence 
  $a_{\alpha-j}\in\R, j\in \N$ such that for any $N\in \N$
  \begin{equation}\label{eq:sigmaN} 
   \sigma_{(N)}(x):=\sigma(x)-\sum_{j=0}^{N-1} a_{\alpha-j}x^{\alpha-j}
  \end{equation}
  lies in $\mathcal{S}^{\alpha-N}(\R_{\geq 1})$. We shall write
  \begin{equation*}
   \sigma  \sim\sum \sigma_{\alpha -j}
  \end{equation*}
  The set of classical symbols on $\R_{\geq 1}$ of order $\alpha$ is written $CS^{\alpha}(\R_{\geq 1})$, and the set of symbols of order 
  less of equal than $\alpha$ is a real vector space. Notice that one needs to be careful, as the difference of two symbols 
  of order $\alpha$ can be of order $\alpha-1$, 
  and we set
  \begin{equation*}
   CS (\R_{\geq 1}):=\sum _{\alpha\in\C}CS^{\alpha}(\R_{\geq 1})
  \end{equation*} 	
  the linear span of classical symbols of all orders. 
  \item Let $k\in\N$ and $\alpha\in\R$. A {\bf log-polyhomogeneous symbol of order $(\alpha,k)$} is a function $f:\R_{\geq1}\longmapsto\R$ 
  such that, for any $x\in\R_{\geq1}$
  \begin{equation} \label{eq:log_symb}
   f(x) = \sum_{l=0}^kf_l(x)\log^l(x)
  \end{equation}
  with $f_l\in CS^\alpha(\R_{\geq1})$. We write $\calp^{\alpha;k}(\R_{\geq1})$ the real vector space spanned by log-polyhomogeneous symbol of 
  order 
  $(\alpha,k)$. We also define
  \begin{equation*}
   \calp^{\alpha;*}(\R_{\geq1}) := \bigcup_{k\in\N}\calp^{\alpha;k}(\R_{\geq1}); \quad \calp^{*;*}(\R_{\geq1}) := \bigcup_{k\in\N}\calp^{*;k}(\R_{\geq1}).   
  \end{equation*}
  with $\calp^{*;k}(\R_{\geq1})$ the linear span of $\R$ of  all $\calp^{\alpha;k}(\R_{\geq1})$ for $\alpha\in\R$.
 \end{enumerate}
\end{defn}
\begin{rk}
 \begin{itemize}
  \item We choose to restrict our discussion to the case $\alpha\in\R$. Complex-valued symbols (resp. classical, log-polyhomogeneous) can 
  be introduced in a similar fashion, however we will not require this level of generality for the construction and study of stuffle arborified 
  zeta 
  values.
  \item We further choose to take our symbols on $\R_{\geq1}$ rather than on $\R^n$ or $\R_{\geq0}$, once again in order to simplify the 
  presentation. 
  \item  Besides these restrictions, our definition of log-polyhomogeneous symbols differs of the one of \cite{MP} as we require the 
  coefficients 
  functions $f_l$ to be classical symbols rather than symbols. This is because these objects appear more naturally in the subsequent 
  developments.
 \end{itemize}
\end{rk}
For any connected subset $I$ of $\R_{\geq1}$ we write $\calp^{\alpha,k}(I)$ for sets of log-polyhomogeneous symbols over $I$. These objects 
have the same definition 
as that of  $\mathcal{P}^{\alpha;k}(\R_{\geq1})$ with $x\in\R_{\geq1}$ replaced by $x\in I$. Notice that 
for two connected 
subsets $I$, $J$ of $\R_{\geq1}$, $I\subseteq J\Rightarrow \calp^{\alpha,k}(J)\subseteq\calp^{\alpha,k}(I)$.

The following Lemma is a direct consequence of the definition of log-polyhomogeneous symbols and standard results of real analysis:
\begin{lem} \label{lem:existence_lim}
 Let $\sigma\in\mathcal{P}^{\alpha;k}(I)$ be a log-polyhomogeneous symbol over $I\subseteq\R_{\geq1}$. If $\alpha<0$, then $\sigma$ admits a finite
 limit in $+\infty$ for any $k$. 
 If $\alpha=0$ then $\sigma$ admits a finite limit in $+\infty$ if and only if $k=0$.
\end{lem}
In the following sections we will focus on the cases $\alpha\in\Z$, which is the reason why we need to introduce log-polyhomogeneous symbols 
rather than working only with classical symbols, as the space of classical symbols with integer orders is not stable under 
taking primitives.

Let us recall important properties of log-polyhomogeneous symbols:
\begin{prop} (\cite[Proposition 2.55]{Pa}) \label{prop:translation}
 Let $\alpha\in\R$ and $k\in\N$. Then for any real number $a<0$, the pull-back translation map $\mathfrak{t}^*_a$ sends  
 $\mathcal{P}^{\alpha;k}(\R_{\geq1})$ to $\mathcal{P}^{\alpha;k}(\R_{\geq1-a})$ i.e.
 for any $\sigma\in\mathcal{P}^{\alpha;k}(\R_{\geq1})$
 \begin{equation*}
  \mathfrak{t}^*_a\sigma:= \left(x\longmapsto\sigma(x+a)\right)
 \end{equation*}
 is an element of $\mathcal{P}^{\alpha;k}(\R_{\geq1-a})$.
\end{prop}
Furthermore, one easily checks the following property:
\begin{prop} \label{prop:bifiltration}
 The space $\calp^{*;*}(\R_{\geq1})$ is a bifiltration, i.e. for any $\alpha,\beta\in\R$ and $k,l\in\N$ we have
\begin{gather*}
  \alpha \leq \beta~\wedge~\alpha-\beta\in\Z \quad \Longrightarrow \quad \calp^{\alpha;k}(\R_{\geq1}) \subseteq \calp^{\beta;k}(\R_{\geq1});\\
  k \leq l\quad \Longrightarrow \quad \calp^{\alpha;k}(\R_{\geq1}) \subseteq \calp^{\alpha;l}(\R_{\geq1}); \\
   \calp^{\alpha;k}(\R_{\geq1}).\calp^{\beta;l}(\R_{\geq1}) \subseteq \calp^{\alpha+\beta;k+l}(\R_{\geq1}).
 \end{gather*}

\end{prop}
As pointed out, our definition of $\calp^{\alpha;k}(\R_{\geq1})$ slightly differs from the one in \cite{Pa} however the same proofs still hold.

Finally, in order to lighten the notations, the spaces $\calp^{\alpha;k}(\R_{\geq1})$ (resp. 
$\calp^{\alpha;*}(\R_{\geq1})$; $\calp^{*;k}(\R_{\geq1})$ and $\calp^{*;*}(\R_{\geq1})$) will be 
written $\calp^{\alpha;k}$ (resp. $\calp^{\alpha;*}$; $\calp^{*;k}$ and $\calp^{*;*}$).

\subsection{Euler-MacLaurin formula}

Following the notations of \cite{CGPZ1} we want to define iterated sums of log-polyhomogeneous symbols:
\begin{equation*}
 S(\sigma)(N):=\sum_{n=1}^N\sigma(n).
\end{equation*}
The Euler-MacLaurin formula (see \cite{H}, formula (13.1.1)) relates such a sum with integral and derivatives of the summed function:
\begin{align}\label{eq:EML_sum}
 S(\sigma)(N) & = \int_1^N\sigma(x)dx + \frac{1}{2}\left(\sigma(N)+\sigma(1)\right) \nonumber\\
              & + \sum_{k=2}^K\frac{B_k}{k!}\,\left(\sigma^{(k-1)}(N)- \sigma^{(k-1)}(1)\right) + \frac{(-1)^{K+1}}{K!}\int_1^N \overline{B_{K}}(x)\,\sigma^{(K)}(x)\, dx
\end{align}
where $\overline{B_k}(x)= B_k\left(x-[x] \right)$ with $[x]$ the integer part of the real number $x$, and $B_k(x)$ the $k$-ht Bernoulli 
polynomial. Notice that \eqref{eq:EML_sum} is independent of the choice of $K\geq2$.

Following once more \cite{MP} and \cite{CGPZ2} we define an operator $P$ acting on smooth functions by
\begin{align} \label{eq:EML}
 P(\sigma)(x) & = \int_1^x\sigma(t)dt + \frac{1}{2}\left(\sigma(t)+\sigma(1)\right) \nonumber\\
              & + \sum_{k=2}^K\frac{B_k}{k!}\,\left(\sigma^{(k-1)}(t)- \sigma^{(k-1)}(1)\right) + \frac{(-1)^{K+1}}{K!}\int_1^x \overline{B_{K}}(t)\,\sigma^{(K)}(t)\, dt.  
\end{align}
Notice that for any $N\in\N^*$ we have $P(\sigma)(N)=S(\sigma)(N)$: $P$ interpolates the discrete sum $S(\sigma)$.

\begin{lem} \label{lem:int_log}
 For any $\alpha\in\R\setminus\Z_{\geq-1}$ and $l\in\N$ we have
 \begin{equation*}
  \int:\calp^{\alpha;l}\longmapsto\calp^{\alpha+1;l} + \calp^{0;0}.
 \end{equation*}
 Furthermore 
 \begin{equation*}
  \int:\calp^{-1;l}\longmapsto\calp^{0;l+1} + \calp^{0;0}.
 \end{equation*}
 In both cases, the operator $\int$ is defined by $(\int f)(x):=\int_1^x\sigma(t)dt$.

\end{lem}

\begin{proof}
 For any $\alpha\in\R\setminus\Z_{\geq-1}$ and $l\in\N$ we have for $\sigma\in\calp^{\alpha;l}$
 \begin{equation*}
  \int_1^x\sigma(t)dt = \sum_{l=0}^k\left[\sigma_{j=0}^{N_l} a_{\alpha-j}^{(l)}\int_1^x t^{\alpha-j}\log^l (t)dt + \int_1^x\sigma_{(N_l),l}(t)\log^l (t)dt\right]
 \end{equation*}
 where we used \eqref{eq:sigmaN} with obvious notations.
 \begin{itemize}
  \item Using that, if $\alpha\in\R\setminus\Z_{\geq-1}$ then $\alpha-j\in\R\setminus\Z_{\geq-1}$ for any $j\in\N$. Then we prove by induction 
  on $l$ that 
  \begin{equation*}
   x\mapsto\int_1^x t^\alpha\log^l(t)dt\in\calp^{\alpha+1;l}+\calp^{0;0}.
  \end{equation*}
  The case $l=0$ is shown by direct integration.
  
  Assuming that this result holds for some $l\in\N$, we have, by integration by part
  \begin{equation*}
   \int_1^x t^\alpha\log^{l+1}(t)dt = \frac{x^{\alpha+1}}{\alpha+1}\log^l(x) + \frac{l}{\alpha+1} \int_1^x t^\alpha\log^l(t)dt.
  \end{equation*}
  By the induction hypothesis, we have $\left(x\mapsto\int_1^x t^\alpha\log^l(t)dt\right)\in\calp^{\alpha+1;l}+\calp^{0;0}$; then Proposition 
  \ref{prop:bifiltration} allows to conclude this induction.
  \item To end this proof, it is now enough to show that, for any $\alpha\in\R\setminus\Z_{\geq-1}$ and $l\in\N$, if 
  $\tau\in\cals^\alpha(\R_{\geq1})$ then it exists $\rho\in\cals^{\alpha+1}(\R_{\geq1})$ and $K\in\R$ such that
  \begin{equation*}
   \int_1^x\tau(t)\log^l (t)dt = \rho(x)\log^l (x) + K.
  \end{equation*}
  Clearly, $\int$ maps smooth functions of $\R_{\geq1}$ to smooth functions of $\R_{\geq1}$. Moreover, using the 
  bound \eqref{eq:symbol} we have
  \begin{equation*}
   \left|\int_1^x\tau(t)\log^l (t)dt\right| \leq C\int_1^x t^{\alpha}\log^l (t)dt
  \end{equation*}
  For some $C\in\R$. Then the induction of the previous point allows us to conclude.
 \end{itemize}
 In the case of $\calp^{-1;l}$ the same arguments hold but the integration (resp. integration by parts) will add one logarithm.
\end{proof}

We further recall a classical Lemma of the theory of log-polyhomogeneous symbols
\begin{lem} \label{lem:deriv_log}
 The differentiation operator $\partial:\sigma\mapsto\sigma'$ sends $\calp^{\alpha;l}$ to $\calp^{\alpha-1;l}$. We also have, for any 
 $\sigma\in\cals^\alpha(\R_{\geq1})$ that $\partial(\sigma\log^l) = \tau\log^l$ for some $\tau\in\cals^{\alpha-1}$.
\end{lem}

We can now state and prove the main Proposition of this subsection.
\begin{prop} \label{prop:prop_P}
 For any $\alpha\in\R\setminus\Z_{\geq-1}$ and $k\in\N$ we have
 \begin{align*}
  P:\mathcal{P}^{\alpha;k} & \longmapsto \mathcal{P}^{\alpha+1;k} + \calp^{0;0}; \\
  P:\mathcal{P}^{-1;k} & \longmapsto \mathcal{P}^{0;k+1} + \calp^{0;0}. \\
 \end{align*}
\end{prop}
\begin{rk}
 This result is a refinement of \cite[Proposition 8]{MP}; which therefore also holds in the slightly  smaller space $\calp^{*;*}$ of the present 
 paper.
\end{rk}

 \begin{proof}
  \begin{enumerate}
   \item 
   Let us start with the case $\alpha\notin\Z_{\geq1}$. Applying Lemmas \ref{lem:int_log} and \ref{lem:deriv_log} in the Euler-MacLaurin 
   formula
   \eqref{eq:EML} we see that its is enough to show that it exists $\tau\in\cals^{\alpha+1-N}$ so that, for a big enough $K$
   \begin{equation*}
    \int_1^x \overline{B_{K}}(t)\,\sigma^{(K)}(t)\, dt = \tau(x)\log^k(x).
   \end{equation*}
   Using that $\overline{B_{K}}(t)$ is bounded and twice Lemma \ref{lem:deriv_log} and Lemma \ref{lem:int_log} we obtain that such a 
   $\tau\in\cals^{\alpha-K+1}$ exists. Therefore we obtain the desired bound by taking $K\geq N$.
   \item The case $\alpha=-1$ is exactly similar, using the second part of Lemma \ref{lem:int_log}.
  \end{enumerate}
 \end{proof}

\subsection{Arborified zeta values as series}

\begin{defn}
 Let $\calr:\N^*\mapsto\calp^{*;*}$ be the map defined by
 \begin{equation*}
  \calr(\alpha):=\left(x\mapsto x^{-\alpha}\right)
 \end{equation*}
 for any $\alpha\in\N^*$. 
\end{defn}
We lift $\calr$ as detailed in Definition \ref{defn:phi_sharp} to obtain the lifted $\calr$-map $\calr^\sharp:\calf_{\N^*}\mapsto\calf_{\calp^{*;*}}$.
\begin{defn}
 For $\lambda\in\{0,-1\}$, we define the operator $\frakS_\lambda$ as $\frakt^*_{\lambda}P$. We further define $\calz_\lambda:=\widehat{\frakS}_\lambda\circ\calr^\sharp:\calf_{\N^*}\mapsto\calp^{*;*}$
\end{defn}
The simple, yet important subsequent Lemma is a consequence of Propositions \ref{prop:prop_P} and \ref{prop:translation}
\begin{lem} \label{lem:ana_prop_frakS}
  For any $\alpha\in\R\setminus\Z_{\geq-1}$, $k\in\N$ and $\lambda\in\{0,-1\}$ we have
 \begin{align*}
  \frakS_\lambda:\mathcal{P}^{\alpha;k} & \longmapsto \mathcal{P}^{\alpha+1;k}(\R_{\geq1-\lambda}) + \calp^{0;0}(\R_{\geq1-\lambda}) \\
  \frakS_\lambda:\mathcal{P}^{-1;k} & \longmapsto \mathcal{P}^{0;k+1}(\R_{\geq1-\lambda}) + \calp^{0;0}(\R_{\geq1-\lambda}). \\
 \end{align*}
\end{lem}
Before stating the main analytic property of $\calz_\lambda$, let us recall some useful notations for forests decorated by $\N^*$:
for any $\N^*$-decorated forest $(F,d)$, we write
 \begin{enumerate}
  \item $|F|:=|V(F)|$ as usual,
  \item $\sharp_n F$ the number of vertices of $F$ decorated by $n\in\N^*$:
  \begin{equation*}
   \sharp_n F:= |\{v\in V(F):d(v)=n\}|.
  \end{equation*}
 \end{enumerate}

\begin{prop} \label{prop:order_calz}
 For any nonempty $\N^*$-decorated tree $(T,d)$ of root $r$ we have
 \begin{equation*}
  \calz_\lambda(T) \in \calp^{1-d(r),\sharp_1 T}(\R_{\geq1-\lambda|F|}) + \calp^{0;0}(\R_{\geq1-\lambda|F|}).
 \end{equation*}
 For any nonempty $\N^*$-decorated tree $(F,d)$ with $F=T_1\cdots T_k$, with $T_i$ nonempty and of root $r_i$ for any $i\in\{1,\cdots k\}$,
 we have
 \begin{equation*}
  \calz_\lambda(F) \in \calp^{1-r_{\rm min}(F),\sharp_1 F}(\R_{\geq1-\lambda|F|}) + \calp^{0;0}(\R_{\geq1-\lambda|F|})
 \end{equation*}
 with $r_{\rm min}(F):=\min_{i=1,\cdots,k}d(r_i)$.
\end{prop}

\begin{proof}
 We prove this result by induction on the number $n=|F|$ of vertices of the forest.
 
 For the case $n=1$ we have, for any $\alpha\in\N^*$
 \begin{equation*}
  \calz_{\lambda}(\tdun{$\alpha$}) = \frakS_\lambda(t\to t^{-\alpha}).
 \end{equation*}
 Since $(t\to t^{-\alpha})$ lies in $\calp^{-\alpha,0}$, the result follows from Lemma \ref{lem:ana_prop_frakS}.
 
 Let us assume this result to hold for any nonempty forests of $n$ or fewer vertices. Let $F$ be a forest of $n+1$ vertices. We distinguish 
 two cases.
 \begin{itemize}
  \item If $F=F_1F_2$ with $F_1$ and $F_2$ non empty then we can write
  \begin{equation*}
   \calz_\lambda(F) = \calz_\lambda(F_1)\calz_\lambda(F_2).
  \end{equation*}
  Indeed $\calz_\lambda$ being the composition of two algebra morphisms, it is an algebra morphism. Then, using the induction hypothesis, the filtration properties 
  of Proposition \ref{prop:bifiltration} (which can be applied since $\calz_\lambda(F_i)$ has integer order) and the remark that $I\subseteq J\Rightarrow \calp^{\alpha,k}(J)\subseteq\calp^{\alpha,k}(I)$, 
  we have 
  \begin{align*}
   \calz_\lambda(F) & \in \calp^{2-r_{\rm min}(F_1)-r_{\rm min}(F_2),\sharp_1 F}(\R_{\geq1-\lambda|F|}) + \calp^{1-r_{\rm min}(F_1),\sharp_1 F_1}(\R_{\geq1-\lambda|F|})  \\ 
    & + \calp^{1-r_{\rm min}(F_2),\sharp_1 F_2}(\R_{\geq1-\lambda|F|}) + \calp^{0;0}(\R_{\geq1-\lambda|F|})
  \end{align*}
  Which, by Proposition \ref{prop:bifiltration} implies the result for $F$ since $1-r_{\rm min}(F_1)\leq0$; $1-r_{\rm min}(F_2)\leq0$ and 
  since each of the orders are integers.
  \item In the case $F=T$ a tree, it exists a nonempty forest $\tilde F$ and a positive natural number $\alpha$ such that 
  $T=B_+^\alpha(\tilde F)$ since $|T|=n+1\geq2$. We then have
  \begin{equation*}
   \calz(T) = \frakS\left(t\to t^{-\alpha}\calz(\tilde F)(t)\right).
  \end{equation*}
  Using the induction hypothesis on $\tilde F$ and the bifiltration property we then obtain
  \begin{equation*}
   t\to t^{-\alpha}\calz(\tilde F)(t) \in \calp^{1-d_{\rm min}-\alpha,\sharp_1\tilde F}(\R_{1-\lambda n}) + \calp^{-\alpha,0}(\R_{1-\lambda n}).
  \end{equation*}
  In the case $\alpha\geq2$, from Lemma \ref{lem:ana_prop_frakS} and the observation that $1-d_{\rm min}\leq0$, we obtain the desired result, once again 
  using Proposition \ref{prop:bifiltration} (since in this case $\sharp_1 T=\sharp_1 \tilde F$), which can be used for the first and the second indices as well.
  
  In the case $\alpha=1$, the same argument holds if one notices that $\sharp_1 T=\sharp_1 \tilde F + 1$.
  
  These two cases allow to conclude the proof of the Theorem.
 \end{itemize}
\end{proof}

\begin{defn}
 A $\N^*$-decorated tree $(T,d)$ is called {\bf convergent} if it is empty or if it has root $r$ and $d(r)\geq2$, i.e. if the decoration of its root is strictly greater than one. A $\N^*$-decorated forest 
 $(F=T_1\cdots T_k,d)$ is called {\bf convergent} if $T_i$ is convergent for each $i\in\{1,\cdots,k\}$. Let $\calf_{\N^*}^{\rm conv}$ be the 
 subalgebra of convergent forests.
 \end{defn}
 It is clear that $\calf_{\N^*}^{\rm conv}$ that is a subalgebra of $\calf_{\N^*}$ since by definition $\emptyset\in\calf_{\N^*}^{\rm conv}$ and $\calf_{\N^*}^{\rm conv}$ is stable by concatenation of forests.
 
 \begin{defnprop} \label{defnprop:arborified_zeta}
 For any convergent $\N^*$-decorated forest $F$ and $\lambda\in\{0,-1\}$, $\calz_\lambda(F)(x)$ admits a finite limit as $x$ goes to infinity. We define the maps 
 $\zeta^T_\stuffle,\zeta^{T,\star}_\stuffle:\calf_{\N^*}^{\rm conv}\mapsto\R$ by
 \begin{equation*}
  \zeta^T_\stuffle(F):=\lim_{x\to\infty}\calz_{-1}(F)(x), \qquad \zeta^{T,\star}_\stuffle(F):=\lim_{x\to\infty}\calz_{0}(F)(x) 
 \end{equation*}
 for any $F\in\calf_{\N^*}^{\rm conv}$.
\end{defnprop}
\begin{proof}
  For any convergent forest $F$ we have $1-d_{\rm min}\leq-1$. Therefore by Lemma \ref{lem:existence_lim} and Proposition
  \ref{prop:order_calz} the limits at $+\infty$ of $\calz_\lambda(F)$ are well-defined and 
  finite for $\lambda\in\{0,-1\}$ provided that $F$ is convergent.
\end{proof}
The definition of AZVs is illustrated by diagram \ref{fig:defn_BZVs}.

Let us notice that the arborified zeta values defined here coincide with the branched zeta values of \cite{CGPZ1}, in the convergent case where the renormalisation scheme reduces to a simple evaluation at $0$ of the regularisation 
parameters. We could therefore have defined AZVs through the regularised branched zeta values of \cite{CGPZ1}. We have not opted for this in order to obtain a self-contained presentation of AZVs. Furthermore, the question of 
renormalisation brings more involved analytic objects, such as meromorphic families of classical symbols, which are unnecessary in the convergent case.

 \begin{figure}[h!] 
  		\begin{center}
  			\begin{tikzpicture}[->,>=stealth',shorten >=1pt,auto,node distance=3cm,thick]
  			\tikzstyle{arrow}=[->]
  			
  			\node (1) {$\calf_{\N^*}^{\rm conv}$};
  			\node (2) [right of=1] {$\R$};
  			\node (3) [below of=1] {$\calf_{\calp^{*;*}}$};
  			\node (4) [right of=3] {$\calp^{*;*}$};

  			\path
  			(1) edge node [above] {$\zeta^T_\stuffle,~\zeta^{T,\star}_\stuffle$} (2)
  			(1) edge node [left] {$\calr^\sharp$} (3)
  			(3) edge node [below] {$\widehat{\frakS}_\lambda$} (4)
  			(1) edge node [left] {$\calz_\lambda$} (4)
  			(4) edge node [right] {${\rm ev}_{\infty}$} (2);
  			
  			\end{tikzpicture}
  			\caption{Definition of arborified zeta values. }\label{fig:defn_BZVs}
  		\end{center}
  	\end{figure}

  	Using the fact that $\calz_\lambda$ and ${\rm ev}_{\infty}$ are both algebra morphisms, we obtain the simple, yet important subsequent Proposition
  	\begin{prop} \label{prop:stuffle_alg_mor}
  	 The maps $\zeta^T_\stuffle$ and $\zeta^{T,\star}_\stuffle$ are algebra morphisms for the concatenation product of trees.
  	\end{prop}
\begin{rk}
\begin{itemize}
 \item Using conical summation techniques,  Zerbini \cite{Ze} was able to relate the values taken by $\zeta^T_\stuffle$ and $\zeta^{T,\star}_\stuffle$ (private communication).
 \item Finally, let us notice that the techniques used above can also be used to define branched Euler sums which correspond (in the words case) to convergent MZVs with some negative arguments. These objects are an active area of 
 research, see for example \cite{WX} for some recent results.
\end{itemize}
\end{rk}

\subsection{Stuffle arborified zetas and  {multiple zetas}}

The construction of the previous subsection can be adapted to build  {multiple zeta} values instead of arborified zetas simply by replacing 
$\calr^\sharp$ by $\calr^\sharp_\calw$ and $\widehat\frakS_\lambda$ by $\widehat\frakS_{\lambda,\calw}$.
\begin{defn}
 For $\lambda\in\{0,-1\}$, let $\calz_{\lambda,\calw}:\calw_{\N^*}\mapsto\calp^{*;*}$ be the operator defined by 
 $\calz_{\lambda,\calw}:=\calr_\calw^\sharp\circ\widehat\frakS_\calw$.
\end{defn}
\begin{thm} \label{thm:calz_words_order}
 For any nonempty word $w=(\omega_1\cdots\omega_n)$, we have
 \begin{equation*}
  \calz_{\lambda,\calw}(w) \in \calp^{1-\omega_1,\sharp_1 w}(\R_{\geq1-\lambda|w|}) + \calp^{0;0}(\R_{\geq1-\lambda|w|}).
 \end{equation*}
\end{thm}
\begin{proof}
 The proof is carried out in the same fashion as the proof of Proposition \ref{prop:order_calz}, with only the case $w=C_+^\omega(\tilde w)$ to be 
 considered.
\end{proof}
\begin{defn}
 A word $w$ written in the alphabet $\N^*$ is called {\bf convergent} if it is empty or if its first letter on the left is 
 greater or equal to $2$.  Let $\calw_{\N^*}^{\rm conv}$ be the 
 linear span of convergent words.
\end{defn}
As the terminology suggests, convergent words (resp. trees) will lie in the convergence domain of MZVs (resp. AZVs).
\begin{lem} \label{lem:subalgebra_shuffle_lambda}
 For any $\lambda\in\R$, $(\calw_{\N^*}^{\rm conv},\emptyset,\shuffle_\lambda)$ is a subalgebra of $(\calw_{\N^*},\emptyset,\shuffle_\lambda)$ .
\end{lem}
\begin{proof}
 By definition $\emptyset\in\calw_{\N^*}^{\rm conv}$. The rest of the proof is easily carried out by induction on $|w|+|w'|$ for two convergent words 
 $w$ and $w'$, using the induction definition \ref{defn:shuffle_prod} of the product $\shuffle_\lambda$.
\end{proof}

 \begin{defnprop} \label{defnprop:zeta_stuffle}
 For any convergent word written in the alphabet $\N^*$ and $\lambda\in\{0,-1\}$, $\calz_{\lambda,\calw}(w)(x)$ admits a finite limit as $x$ 
 goes to infinity. We define the maps 
 $\zeta_\stuffle,\zeta^{\star}_\stuffle:\calw_{\N^*}^{\rm conv}\mapsto\R$ by
 \begin{equation*}
  \zeta_\stuffle(F):=\lim_{x\to\infty}\calz_{-1,\calw}(F)(x), \qquad \zeta^{\star}_\stuffle(F):=\lim_{x\to\infty}\calz_{0,\calw}(F)(x) 
 \end{equation*}
 for any $w\in\calw_{\N^*}^{\rm conv}$.
\end{defnprop}
\begin{proof}
 The proof of the statement is similar to the proof of Definition-Proposition \ref{defnprop:arborified_zeta}.
\end{proof}
Notice that this approach to  {multiple zeta} values easily yields back well-known (see for example \cite{Wa}) results for these numbers:
\begin{prop} \label{prop:zeta_stuffle_mor}
 The map $\zeta_\stuffle$ is a algebra morphism for the stuffle product $\stuffle=\shuffle_1$. Furthermore the map $\zeta^\star_\stuffle$ is 
 an algebra morphism for the anti-stuffle product $\shuffle_{-1}$.
\end{prop}
\begin{proof}
 By Examples \ref{ex:RBmap}, \ref{ex:RB_sum} (resp. \ref{ex:RBmap}, \ref{ex:RB_sum_star}) we know that $\frakS_{-1}$ (resp. $\frakS_0$), when restricted 
 to $\N^*$, is a Rota-Baxter map of weight $+1$ (resp $-1$). One can then apply Theorem \ref{thm:flattening} \ref{thm:iii}  to get that 
 $N\to\widehat\frakS_{-1,\calw}(N)$ (resp. $N\to\widehat\frakS_{0,\calw}(N)$), is an algebra morphism for the stuffle 
 (resp. anti-stuffle) product. Taking the limit $N\to\infty$ gives to the result.
\end{proof}
Before stating the main result of this subsection, let us state a simple Lemma.
\begin{lem} \label{lem:flattening_conv}
 For any $\lambda\in\R$, the flatting map $fl_\lambda$ maps convergent forests to convergent words.
\end{lem}
\begin{proof}
 One can perform an easy proof by induction on the number of vertices of the convergent forest. For the empty forest, we have 
 $fl_\lambda(\emptyset)\in\calw_{\N^*}^{\rm conv}$. For a nonempty convergent tree $T=B_+^{n\geq2}(F)$ for some forest $F$. Then 
 $fl_\lambda(T)=C_+^{n\geq2}(fl_\lambda(F))\in\calw_{\N^*}^{\rm conv}$. Finally for the case $F=F_1F_2$, we have that $F_1$ and $F_2$ are 
 convergent if $F$ is and therefore $fl_\lambda(F) = fl_\lambda(F_1)\shuffle_\lambda fl_\lambda(F_2)\in\calw_{\N^*}^{\rm conv}$ by the 
 induction hypothesis and Lemma \ref{lem:subalgebra_shuffle_lambda}.
\end{proof}
We can now relate these two constructions.
\begin{thm} \label{thm:main_result_stuffle}
 For any convergent forest $F$, the convergent  arborified zeta value $\zeta^T_\stuffle(F)$ (resp. $\zeta^{T,\star}_\stuffle(F)$) is a finite 
 linear combination of convergent  {multiple zeta} values $\zeta_\stuffle(w)$ (resp. $\zeta^{\star}_\stuffle(w)$) with rational coefficients and can be written as a finite 
 linear combination of  {multiple zeta} values with integer coefficients.
\end{thm}
\begin{proof}
 Let $F$ be any convergent forest.
 
 Once again, by Examples \ref{ex:RBmap}, \ref{ex:RB_sum} (resp. \ref{ex:RBmap}, \ref{ex:RB_sum_star}) we know that $\frakS_{-1}$ (resp. 
 $\frakS_0$), when restricted to $\N^*$, is a Rota-Baxter map of weight $+1$ (resp $-1$). Applying Theorem \ref{thm:flattening} \ref{thm:ii} we
 get that $N\to\widehat\frakS_{-1}(F)(N)$ (resp. $N\to\widehat\frakS_{0}(F)(N)$), when restricted to $\N^*$, factorises through words, i.e. 
 that $\widehat\frakS_\lambda(F)(N) = \widehat\frakS_{\lambda,\calw}\circ fl_\lambda(F)(N)$ for any $N$ in $\Z_{\geq|F|+1}$ and 
 $\lambda\in\{0,-1\}$.
 
 Then by Lemma \ref{lem:flattening_conv} we can take the limit $N\to\infty$ in both sides which gives the result. 
 We obtain a finite sum with integer coefficients thanks to Proposition \ref{prop:finite_Q_sum_flattening}.
\end{proof}
To summarize the main results of this Section, the definitions and relations between branched zetas and  {multiple zetas} can be summarized as the 
commutativity of diagram \ref{fig:MZV_AZV}.

 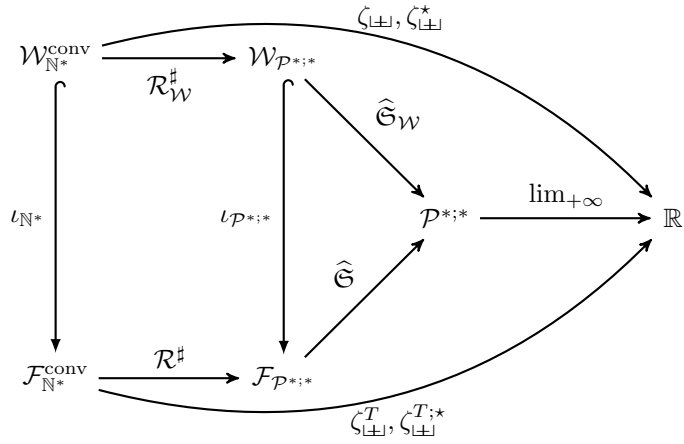
\begin{figure}[h!] 
  		\begin{center}
  			\begin{tikzpicture}[->,>=stealth',shorten >=1pt,auto,node distance=3cm,thick]
  			\tikzstyle{arrow}=[->]
  			
  			\node (1) {$\calw_{\N^*}^{\rm conv}$};
  			\node (2) [right of=1] {$\calw_{\calp^{*;*}}$};
  			\node (4) [below right of=2] {$\calp^{*;*}$};
  			\node (5) [below left of=4] {$\calf_{\calp^{*;*}}$};
  			\node (7) [left of=5] {$\calf_{\N^*}^{\rm conv}$};
  			\node (8) [right of=4] {$\R$};

  			\path
  			(1) edge node [below] {$\calr^\sharp_\calw$} (2)
  			(2) edge node [above right] {$\widehat\frakS_\calw$} (4)
  			(4) edge node [above] {$\lim_{+\infty}$} (8)
  			(7) edge node [above] {$\calr^\sharp$} (5)
  			(5) edge node [above left] {$\widehat\frakS$} (4);

  			\draw [right hook-latex] (1) -- node[left] {$\iota_{\N^*}$} (7);
  			\draw [right hook-latex] (2) -- node[left] {$\iota_{\calp^{*;*}}$} (5);
  			
  			\draw (1) to[bend left] node[above]{$\zeta_\stuffle,\zeta_\stuffle^\star$} (8);
  			\draw (7) to[bend right] node[below]{$\zeta^T_\stuffle,\zeta_\stuffle^{T;\star}$} (8);
  			\end{tikzpicture}
  			\caption{ {Multiple zetas} and arborified zetas.}\label{fig:MZV_AZV}
  		\end{center}
  	\end{figure}
 {Furthermore, the relationships between convergent ordinary zeta values, arborified convergent zeta values and the flattening map is illustrated in the commutative diagram \ref{fig:flattening_zeta}.
\begin{figure}[h!]
 \begin{center}
  \begin{tikzpicture}[->,>=stealth',shorten >=1pt,auto,node distance=3cm,thick]
  			\tikzstyle{arrow}=[->]
  			
  			\node (1) {$\calw_{\N^*}^{\rm conv}$};
  			\node (2) [right of=1] {$\R$};
  			\node (3) [below of=1] {$\calf_{\N^*}^{\rm conv}$};
  			
  			\path
  			(1) edge node [above] {$\zeta_\stuffle$} (2)
  			(3) edge node [left] {$fl_1$} (1)
  			(3) edge node [right] {$\zeta^T_\stuffle$} (2);
  \end{tikzpicture}
  \caption{Convergent zetas and flattening.} \label{fig:flattening_zeta}
 \end{center}
\end{figure}
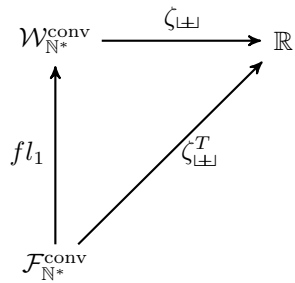
The same diagram holds with $\zeta_\stuffle$ (resp. $fl_1$ and $\zeta^T_\stuffle$) replaced by $\zeta_\stuffle^\star$ (resp. $fl_{-1}$ and $\zeta_\stuffle^{T;\star}$). }

\begin{rk}
 Before moving on to shuffle arborified zeta values, let us mention that an integral representation of stuffle zeta values was proposed in \cite{M}, but it does not 
 preserves the structure of trees. Such an integral representation was also presented in \cite{Ze}.
\end{rk}

\section{Shuffle arborified zeta values} \label{section:shuffle}

\subsection{Chen integrals and arborification}

In \cite{Ch} iterated integrals are recursively defined. One way to define them is as a map $Ch:\calw_X\mapsto \cali(I)$; where $I=[a,b]$ is a
closed interval, $\cali(I)$ is the set of continuous, integrable functions over $I$ and $X=\{f_1,\cdots,f_N\}$ is a finite subset of $\cali(I)$.
In \cite{Ch} this recursive definition goes as follows:
\begin{defnprop} \label{defnprop:Chen_int}
 Let $I=[a,b]$ be a closed interval in $\R$ and $X=\{f_1,\cdots,f_N\}$ be a finite set with $f_i:I\mapsto\R$ smooth, continuous functions over 
 $I$. $Ch:\calw_X\mapsto \cali(I)$ is the linear map, whose action on the basis elements of $\calw_X$ is recursively defined by
 \begin{align*}
  Ch(\emptyset) & := \left(x\mapsto1~\forall x\in I\right)=:{\bf 1}, \\
  Ch((f_i)\sqcup w) & := \left(x\mapsto\int_a^x f_i(t)Ch(w)(t)dt~\forall x\in I\right)
 \end{align*} 
 for any $f_i$ in $X$ and $w$ in $\calw_X$.
\end{defnprop}
\begin{proof}
 In order to prove that this definition is consistent, one has to prove that $Ch(w)$ lies in $\cali(I)$ for any word $w$ in $\calw_X$. 
 We show by induction that it exists $M\in\R$ such that for any $x$ in $I$ and $w$ in $\calw_X$, one has
 \begin{equation*}
  |Ch(w)(x)| \leq \left(M|b-a|\right)^{|w|}.
 \end{equation*}
 This can easily be done by induction on the length of the word $w$. For a word of length $0$, it is trivially true. 
 Now, since each $f_i$ in $X$ is continuous over the closed interval $I$, $f_i$ is bounded by some constant $M_i$. Let 
 $M:=\max_{i=1,\cdots,n}M_i$. If $|Ch(w)(x)| \leq \left(M|b-a|\right)^{|w|}$ holds for every word $w$ of length $n\geq1$, then one has for any 
 $f_i$ in $X$
 \begin{equation*}
  |Ch((f_i)\sqcup w)(x)| \leq \int_a^x|f_i(t)Ch(w)(t)|dt \leq \left(M|b-a|\right)^{|w|+1}.
 \end{equation*}
 This shows that $|Ch(w)(x)| \leq \left(M|b-a|\right)^{|w|}$ for any $x$ in $I$ and $w$ in $\calw_X$. Continuity of $Ch(w)$ follows from the 
 standard theorems of integration theory.
\end{proof}
We denote by $Ch_X(I)$ the algebra (for the  {pointwise} product) freely generated by the image of $Ch$. It admits ${\bf 1}=Ch(\emptyset)$ as a unit.

This approach shows that Chen integrals are the elements of the image of a map $\widehat{\cali_a^b}_\calw$ going from $\calw_X$ to a subset of 
$\cali(I)$, where $\cali_a^b$ is the integration map defined by $\cali_a^b(f)(x):=\int_a^x f(t)dt$ for $x$ in $[a,b]$. Indeed, standard results of integration theory 
states that, since the interval $I$ is closed, the image of $\cali(I)$ under the map $\cali_a^b$ lies in $\cali(I)$. This suggest a natural 
generalisation of Chen integrals to arborified Chen integrals.

First, observe that if $X\subseteq Y$, then $\calf_X$ forms a subalgebra of $\calf_Y$ (since $\emptyset$ is an element of $\calf_X$ for any set 
$X$).
\begin{defn} \label{defn:arborified_chen_int}
 Let $I=[a,b]$ and $\cali(I)$ be as above. Let $\cali_a^b:\cali(I)\mapsto\cali(I)$ be 
 the integration map defined as above by $\cali_a^b(f)(x):=\int_a^x f(t)dt$. Let $X=\{f_1,\cdots,f_N\}$ is a finite subset of $\cali(I)$. 
 {\bf arborified Chen integrals} are elements of the image of $\widehat{\cali_a^b}:\calf_X\mapsto\cali(I)$.
\end{defn}
\begin{rk}
 One could recursively prove that $\widehat{\cali_a^b}$ is well-defined as in the case of words, that is that the image of $\widehat{\cali_a^b}$ 
 is indeed in $\cali(I)$. However by Definition \ref{defn:phi_hat} we already have this result since $\cali_a^b(\cali(I))\subseteq\cali(I)$.
\end{rk}
\begin{prop} \label{prop:chen_int_arbo_words}
 Any arborified Chen  {integral} is a finite sum of Chen iterated integrals with rational coefficients.
\end{prop} 
\begin{proof}
 This follows from the observation above that Chen iterated integrals are elements of the image of $\widehat{\cali_a^b}_\calw$. Then by Example 
 \ref{ex:RBmap} \ref{ex:RB_integration} we can apply Theorem \ref{thm:flattening} to $\widehat{\cali_a^b}$ which gives 
 $\widehat{\cali_a^b}=\widehat{\cali_a^b}_\calw\circ fl_0$. For any finite forest $F$, $fl_0(F)$ is a finite sum of words with rational coefficients by 
 Proposition \ref{prop:finite_Q_sum_flattening}, thus 
 we obtain the result.
\end{proof}

\subsection{Arborified polylogarithms} 

We specialise the construction of the previous subsection to the case $X=\{\sigma_x,\sigma_y\}$ with $\sigma_x(t):=1/t$ and 
$\sigma_y(t):=1/(1-t)$. Furthermore the above construction is carried out on $I=[\epsilon,z]$ with $0<\epsilon<z<1$.

However arborified polylogarithms (resp.  arborified shuffle zeta values) should be defined (up to convergence issues) as a map 
$Li^T:\calf_{\{x,y\}}\mapsto\Omega$ (resp. $\zeta^T_\shuffle:\calf_{\{x,y\}}\mapsto\R$) for  {a suitable space $\Omega$ of functions}. In order 
to build such a map, we follow the same strategy as for stuffle zeta values. Let $\calr_{\{x,y\}}:\{x,y\}\mapsto\{\sigma_x,\sigma_y\}$ defined 
by $\calr_{\{x,y\}}(\varepsilon)=\sigma_\varepsilon$ for $\varepsilon$ in $\{x,y\}$. 
\begin{defn}
 A forest $F$ in $\calf_{\{x,y\}}$ is called {\bf semiconvergent} if each of its leaves and branching vertices are decorated by $y$ and {\bf convergent} if it is 
 semiconvergent and each of its roots are decorated by $x$. The linear span of semiconvergent forests is denoted by $\calf^{\rm semi}_{\{x,y\}}$ 
 and the linear span of convergent forests is denoted by $\calf^{\rm conv}_{\{x,y\}}$.
 
 Similarly, a word $w$ in $\calw_{\{x,y\}}$ is called {\bf semiconvergent} if it is empty or if it ends by $y$ and {\bf convergent} if it is empty or if it ends by $y$ and starts by $x$. 
 We write $\calw^{\rm semi}_{\{x,y\}}$ and $\calw^{\rm conv}_{\{x,y\}}$ the linear span of semiconvergent and convergent words respectively.
\end{defn}
\begin{lem}
 $\calf^{\rm semi}_{\{x,y\}}$ and $\calf^{\rm conv}_{\{x,y\}}$ are  {subalgebras} of $\calf_{\{x,y\}}$ for the concatenation of forests; $\calw^{\rm semi}_{\{x,y\}}$ and $\calw^{\rm conv}_{\{x,y\}}$
  are subalgebras of $\calw_{\{x,y\}}$ for the shuffle product $\shuffle$.
\end{lem}
\begin{proof}
 As before, the result trivially holds for forests. For words the result follows from the fact that for any two words $w$ and $w'$, then we can write $w\shuffle w'=\sum_i w_i$ and the first (resp. last) letter 
 of each $w_i$ is the first (resp. last) letter of $w$ or $w'$.
\end{proof}
In order to state an important result of this Section, let us recall the definition of multiple polylogarithms.
\begin{defn} \label{defn:multiple_polylogs}
 Let $w\in\calw_{\{x,y\}}^{\rm semi}$ be a word either empty or whose last letter is $y$. The {\bf single variable multiple polylogarithm} (shortened in multiple polylogarithm in what follows) attached to $w$ is defined by
 \begin{align*}
   Li_w(z):=\begin{cases}
             & 0 \quad {\rm for } \quad z=0, \\
             & \lim_{\epsilon\to0}\widehat{\cali_\epsilon^z}_\calw\left(\calr_{\{x,y\}}^{\sharp,\calw}(w)\right) \quad {\rm for } \quad z\in]0,1[.
            \end{cases}
  \end{align*}
  We write $Li:\calw_{\{x,y\}}^{\rm semi}\longrightarrow\mathcal{C}^\infty([0,1[)$ the map which, to such a word, associates the map $z\mapsto Li_w(z)$.
\end{defn}
The existence of the limit for semiconvergent words and the fact that $Li_w$ is a smooth map are well-known results of polylogarithms theory, see for example \cite{Br}.

\begin{defnprop} \label{defnprop:arbo_polylogs}
 For any $z\in]0,1[$ and semiconvergent $F$ the limit
 \begin{equation*}
  Li^T_F(z) := \lim_{\epsilon\to0}\widehat{\cali_\epsilon^z}\left(\calr_{\{x,y\}}^\sharp(F)\right)
 \end{equation*}
 exists. Setting $Li^T_F(z)=0$ we obtain a map 
 \begin{align*}
  Li^T_F:[0,1[ & \mapsto\R \\
           z & \mapsto Li^T_F(z)
 \end{align*}
 is called the {\bf arborified polylogarithm} associated to the semiconvergent forest $F$. The {\bf arborified polylogarithm map} is defined 
 by its action $Li^T:F\mapsto Li^T_F$ on any semiconvergent forest $F$.
\end{defnprop}
\begin{proof}
 One needs to prove the existence of the limit for any $z\in[0,1[$. It follows from Example \ref{ex:RBmap}, \ref{ex:RB_integration} that we can 
 apply Theorem \ref{thm:flattening} with $\lambda=0$. Furthermore, one easily shows by induction that the image of a semiconvergent forest under 
 the map $fl_0$ is a finite sum of words, each ending with $y$. As stated above, it is a well-known fact (see for example \cite{Br}) that for such a word $w$,
 the limit
 \begin{equation*}
  \lim_{\epsilon\to0}\widehat{\cali_\epsilon^z}_\calw\left(\calr_{\{x,y\}}^{\sharp,\calw}(w)\right)
 \end{equation*}
 exists. The result then follows.
\end{proof}
\begin{lem}\label{lem:flattening_conv_xy}
 The flattening map of weight $0$ maps $\calf_{\{x,y\}}^{\rm semi}$ (resp. $\calf_{\{x,y\}}^{\rm conv}$) to $\calw_{\{x,y\}}^{\rm semi}$ (resp. $\calw_{\{x,y\}}^{\rm conv}$).
\end{lem}
\begin{proof}
 The proof is similar to the proof of Lemma \ref{lem:flattening_conv}.
\end{proof}

\begin{thm} \label{thm:arborified_polylogs}
 For any semiconvergent forest $F$, the arborified polylogarithm associated to $F$ enjoys the following properties
 \begin{enumerate}
  \item it is a finite sum of multiple polylogarithms with rational coefficients that can be written as a finite 
 linear combination of multiple polylogarithms with integer coefficients;
  \item it is a smooth map on $[0,1[$;
  \item The arborified polylogarithm map $Li^T:F\mapsto Li^T_F$ is a algebra morphism for the concatenation of trees and the  {pointwise} product of functions.
 \end{enumerate}
\end{thm}
\begin{proof}
 \begin{enumerate}
  \item The proof of this result is the same as the proof of Theorem \ref{thm:main_result_stuffle}, using Example \ref{ex:RBmap}, \ref{ex:RB_integration}. The limits in the definition of $Li_w$ are well-defined 
  by Lemma \ref{lem:flattening_conv_xy}.
  \item The second point follows from the first, since multiple polylogarithms are smooth maps of $[0,1[$.
  \item This follows from the fact that $Li^T$ is the composition of $\widehat{\cali_\epsilon^z}_\calw$ and $\calr_{\{x,y\}}^{\sharp,\calw}$, which are both algebra morphisms.
  Furthermore, if $FF'$ is a semiconvergent forest, then $F$ and $F'$ are two semiconvergent forests. Then $\lim_{\epsilon\to0}\widehat{\cali_\epsilon^z}\left(\calr_{\{x,y\}}^\sharp(F)\right)$
  and $\lim_{\epsilon\to0}\widehat{\cali_\epsilon^z}\left(\calr_{\{x,y\}}^\sharp(F')\right)$ exist and their product is equal to 
  \begin{equation*}
   \lim_{\epsilon\to0}\left[\widehat{\cali_\epsilon^z}\left(\calr_{\{x,y\}}^\sharp(F)\right)\widehat{\cali_\epsilon^z}\left(\calr_{\{x,y\}}^\sharp(F)\right)\right]
  \end{equation*}
  Thus the limit in the definition of $Li^T$ is also an algebra morphism and $Li^T$ is an algebra morphism as stated.
 \end{enumerate}

\end{proof}
As in the case of the stuffle branched zeta values, one can use this framework to provide a new proof that multiple polylogarithms are algebra morphisms for the shuffle product.
\begin{prop} \label{prop:polylogs_shuffle_mor}
 The map $Li:\calw_{\{x,y\}}^{\rm semi}\mapsto\mathcal{C}^{\infty}([0,1],\R)$ is an algebra morphism for the shuffle product.
\end{prop}
\begin{proof}
 The proof follows the same steps as the proof of Proposition \ref{prop:zeta_stuffle_mor}.
\end{proof}

\subsection{Arborified zeta values as integrals}

It is well-known (see for example \cite{Br} or \cite{Wa}) that for a convergent word $w$ 
the limit $\lim_{z\to1}Li_w(z)$ exists. This allows the following definition.
\begin{defn} \label{defn:shuffle_MZVs}
 Let $w\in\calw_{\{x,y\}}$ be a word starting with $x$ and ending with $y$, then the {\bf shuffle multiple zeta value} associated to $w$ is the real number
 \begin{equation*}
  \zeta_\shuffle(w):=\lim_{z\to1}Li_w(z).
 \end{equation*}
 We write $\zeta_\shuffle$ the map defined by
 \begin{align*}
  \zeta_\shuffle:\{\emptyset\}\bigcup\left((x)\sqcup\calw_{\{x,y\}}\sqcup(y)\right) \longrightarrow \R & \\
  w \longmapsto \zeta_\shuffle(w & ).
 \end{align*}
\end{defn}
This definition can easily be generalised to trees, thanks to the following Lemma.
\begin{lem} \label{lem:lim_z_1_forests}
 For any convergent forest $F\in\calf_{\{x,y\}}^{\rm conv}$, the limit $\lim_{z\to1}Li^T_F(z)$ exists.
\end{lem}
\begin{proof}
 Let $F\in\calf_{\{x,y\}}^{\rm conv}$. If $F=\emptyset$, then $fl_0(F)=\emptyset$ and the result trivially holds. Otherwise $fl_0(F)\in(x)\sqcup\calw_{\{x,y\}}\sqcup(y)$. Indeed, we can write $fl_0(F)=\sum_{i\in I}w_i$ for some finite set $I$. Then each $w_i$ has the decoration of a root of $F$ (thus a $x$) 
 as its first letter; and the decoration of a leaf of $F$ (thus a $y$) as its last letter. This is shown ad absurdum: if we have $w_i=(y)\sqcup\tilde w$, then a vertex decorated by $y$ was not above all roots of $F$, since 
 every roots of $F$ is decorated by $x$. This is a contradiction. The same argument shows that no $w_i$ cannot end with an $x$. The result then follows from Theorem \ref{thm:arborified_polylogs} and the observation above that 
 $\lim_{z\to1}Li_w(z)$ exists for any 
 $w\in (x)\sqcup\calw_{\{x,y\}}\sqcup(y)$.
\end{proof}
This allows the following definition.
\begin{defn} \label{defn:shuffle_AZVs}
 For any convergent forest $F\in\calf_{\{x,y\}}^{\rm conv}$ the  {corresponding} {\bf shuffle arborified zeta values} is defined as
 \begin{equation*}
  \zeta_\shuffle^T(F) := \lim_{z\to1}Li_F^T(z).
 \end{equation*}
 We write $\zeta_\shuffle^T$ the map defined by
 \begin{align*}
  \zeta_\shuffle^T:\calf_{\{x,y\}}^{\rm conv} \longrightarrow & \R \\
  F \longmapsto \zeta_\shuffle^T&(F).
 \end{align*}
\end{defn}
Shuffle arborified zeta values enjoy the following properties.
\begin{thm} \label{thm:main_result_shuffle}
 For any convergent forest $F\in\calf_{\{x,y\}}^{\rm conv}$,  {the shuffle arborified zeta values} $\zeta_\shuffle^T(F)$ is a finite sum of  {multiple zeta} values with rational coefficients that can be written as a finite 
 sum of  {multiple zeta} values with integer coefficients. Furthermore the map $\zeta_\shuffle^T:\calf_{\{x,y\}}^{\rm conv} \longrightarrow \R$
 is an algebra morphism for the concatenation product of trees.
\end{thm}
\begin{proof}
 This Theorem is a consequence of Theorem \ref{thm:arborified_polylogs} applied to convergent forests, for which one can take the limit $z\to1$ according to Lemma \ref{lem:lim_z_1_forests}.
\end{proof}
\begin{rk}
 One could prove this theorem along the steps of the proof of Theorem \ref{thm:main_result_stuffle}. This is true in many occurrences throughout this section. Hence Sections \ref{section:stuffle} and 
 \ref{section:shuffle} present two different (however equivalent) ways of building branched objects.
\end{rk}
We conclude this section by pointing out that, as a consequence of our previous results, we also have shown that  {multiple zetas} as iterated integrals are algebra morphism for the shuffle product.
\begin{prop} \label{prop:shuffle_zeta_alg_mor_shuffle}
 The map $\zeta_\shuffle:\{\emptyset\}\bigcup\left((x)\sqcup\calw_{\{x,y\}}\sqcup(y)\right) \longrightarrow \R$ is an algebra morphism for the shuffle product.
\end{prop}
\begin{proof}
 The proof follows by using Proposition \ref{prop:polylogs_shuffle_mor} on words in $\{\emptyset\}\bigcup\left((x)\sqcup\calw_{\{x,y\}}\sqcup(y)\right)$ and taking the limit $z\to1$.
\end{proof}

\section{Shuffles and arborified zetas} \label{section:shuffle_tree}

\subsection{Shuffles on trees}

We have seen that Theorems \ref{thm:main_result_stuffle} and \ref{thm:main_result_shuffle}, together with the fact that the maps $\zeta^T_\stuffle$ and 
$\zeta^T_\shuffle$ are algebra morphisms for the concatenation product of trees, allow to prove that the maps $\zeta_\stuffle$ and $\zeta_\shuffle$ are morphisms for the stuffle and shuffle product respectively. One 
can therefore interpret these theorems as the generalisation to trees of the stuffle and shuffle products. However, this interpretation is not entirely satisfactory as the flattening maps \eqref{defn:flattening} appearing in Theorems
\ref{thm:main_result_stuffle} and \ref{thm:main_result_shuffle} obfuscates the tree structure.

 {Another issue with seeing the concatenation product as the generalisation to trees of the shuffle and stuffle products is that it does not impose relations of the AZVs. This is due to the fact that the AZV 
associated to a non connected forest $F_1F_2$ can be \emph{defined} as the product of the AZVs associated to $F_1$ and $F_2$ (see Equation \eqref{eq:prod_hat_phi}). On the contrary, the shuffle and stuffle product 
do impose relation among MZVs.}

Hence, in order to rightfully claim to have generalised the shuffle and stuffle products to trees, one should find suitable  products $\star:\calf_\Omega\otimes\calf_\Omega\longrightarrow\calf_\Omega$ for which  the maps 
$\zeta_\stuffle^T$ and 
$\zeta_\shuffle^T$ are  {non trivial} algebra morphisms.  {This is achieved in Theorem \ref{thm:AZV_alg_morphism_shuffle}. Furthermore, Corollary \ref{coro:relation_nonasso} 
gives relations among the AZVs which have no equivalent among MZVs.} We build these products using the same recursive recipe as for  the shuffle products of Definition \ref{defn:shuffle_prod}, with the concatenation of words 
replaced by the grafting.
\begin{defn} \label{defn:shuffle_tree}
 Let $\Omega$ be a set (resp. $(\Omega,.)$ be a commutative semigroup and $\lambda\in\R$). The {\bf shuffle product on trees} $\shuffle$ (resp. the {\bf $\lambda$-shuffle product on trees} $\shuffle_\lambda$) of two forests $F$ and 
 $F'$ is defined recursively on $|F|+|F'|$. 
 
 If $|F|+|F'|=0$ (and thus $F=F'=\emptyset$), we set $\emptyset\shuffle\emptyset=\emptyset\shuffle_\lambda\emptyset=\emptyset$.
 
 For $N\in\N$, assume the shuffle (resp. $\lambda$-shuffle) products of forests has been defined on every forests $f,f'$ such that $|f|+|f'|\leq N$. Then for any two forests $F,F'$ such that $|F|+|F'|=N+1$;
 \begin{itemize}
  \item (Unit) if $F'=\emptyset$, set $\emptyset\shuffle F=F\shuffle\emptyset=F$; and $F\shuffle_\lambda\emptyset=\emptyset\shuffle_\lambda F=F$.
  \item (Compatibility with the concatenation product) if $F$ or $F'$ is not a tree, then we can write $F$ and $f$ uniquely as a concatenation of trees: $F=T_1\cdots T_k$ and $F'=t_1\cdots t_n$ with the $T_i$s and $t_j$s 
  nonempty, $k+n\geq3$ and set 
  \begin{equation*}
   F\shuffle F' = \frac{1}{kn}\sum_{i=1}^k\sum_{j=1}^n\left((T_i\shuffle t_j)T_1\cdots\widehat{T_i}\cdots T_n t_1\cdots\widehat{t_j}\cdots t_k\right)
  \end{equation*}
  (resp. 
  \begin{equation*}
   F\shuffle_\lambda F' = \frac{1}{kn}\sum_{i=1}^k\sum_{j=1}^n\left((T_i\shuffle_\lambda t_j)T_1\cdots\widehat{T_i}\cdots T_n t_1\cdots\widehat{t_j}\cdots t_k\right)~),
  \end{equation*}
  where $T_1\cdots\widehat{T_i}\cdots T_n$ stands for the concatenation of the trees $T_1,\cdots,T_n$ without the tree $T_i$.
  \item (Compatibility with the grafting) if $F=T = B_+^a(f)$ and $F'=T'=B_+^{a'}(f')$ are two nonempty trees, we set 
  \begin{equation*}
   T\shuffle T' = B_+^a(f\shuffle T') + B_+^{a'}(T\shuffle f')
  \end{equation*}
  (resp. 
  \begin{equation*}
   T\shuffle_\lambda T' = B_+^a(f\shuffle_\lambda T') + B_+^{a'}(T\shuffle_\lambda f') + \lambda B_+^{a.a'}(f\shuffle_\lambda f')~).
  \end{equation*}
 
 \end{itemize}
\end{defn}

\begin{rk}
 \begin{itemize}
  \item The well-definedness of the products $\shuffle$, $\shuffle_\lambda$ follows from the fact that any forest can be uniquely written (up to permutation) as an iteration of concatenations and graftings.
  \item As before (Remark \ref{rk:set_monoid}), we notice that $\shuffle_0=\shuffle$. We nevertheless make a distinction between the cases $\lambda=0$ and $\lambda\neq0$ as in the former case, the set $\Omega$ is not 
  required to have a semigroup structure. We will however treat the $\shuffle$ product as a special case of $\shuffle_\lambda$, keeping in mind that when dealing with 
  $\lambda=0$ (so with the shuffle product), we will always implicitly allow the decoration set to not have a semigroup structure. 
  \item We use the same symbols for shuffles on trees and shuffles and words, as whether we are working with words or with trees shall be clear from context. So, as for words, we write $\stuffle:=\shuffle_1$ the 
  {\bf stuffle product on trees} and $\shuffle_{-1}$ the {\bf anti-stuffle product on trees}.
 \end{itemize}

\end{rk}

\begin{ex}
 Here are examples of stuffle of trees with $(\Omega,.)=(\N,+)$:
 \begin{gather*}
  \tdun{n}\tdun{m}\shuffle_\lambda\tdun{p} = \frac{1}{2}\Big(\tdun{n}\left(\tddeux{~m}{~p} + \tddeux{~p}{~m} + \lambda\tdun{m+p}\quad\right) + \tdun{m}\left(\tddeux{~n}{~p} + \tddeux{~p}{~n} + \lambda\tdun{n+p}\quad\right)\Big) \\
  \tdtroisun{q}{~m}{n}\shuffle_\lambda\tdun{p} = \tdquatrequatre{p}{q}{n}{m} + \frac{1}{2}\Big(\tdquatretrois{q}{m}{p}{n} + \tdquatretrois{q}{p}{m}{n} + \lambda\tdtroisun{q}{m+p}{n} +\tdquatrequatre{q}{p}{n}{m} + \tdquatretrois{q}{n}{p}{m~} + \tdquatretrois{q}{p}{n}{m~} + \lambda\tdtroisun{q}{n+p}{m~}\Big) + \lambda\tdtroisun{q+p}{m}{n}. 
 \end{gather*}
 On the ground of these intermediate computations, one can compute more involved shuffles of trees. However due to their length,   we will not write down the result here. For example the shuffle 
 $\tdtroisun{r}{m}{n}\shuffle_\lambda\tdtroisun{s}{q}{p}$ is a sum of forty-four trees, of which twenty have six vertices, twenty have five vertices, and four have four vertices.
\end{ex}

We turn now our attention to the structures inherited by $\calf_\Omega$ from these products.

\begin{prop} \label{prop:shuffles_trees}
 Let $\lambda\in\R^*$ and $(\Omega,.)$ be a commutative semigroup; or $\lambda=0$ and $\Omega$ a set. Then $(\calf,\shuffle_\lambda,\emptyset)$ is an nonassociative, commutative, unital algebra.
\end{prop}
\begin{proof}
 The case $\lambda=0$ is a consequence of the more general case as every undefined product in $\Omega$ disappear if $\lambda$ is set to $0$. Therefore we will only explicitly work out the case $\lambda\neq0$. 
 
 Let $(\Omega,.)$ be a commutative semigroup and $\lambda\in\R$.
 \begin{enumerate}
  \item By definition of $\shuffle_\lambda$, $\emptyset\shuffle_\lambda F=F\shuffle_\lambda\emptyset=F$ for any $F\in\calf_\Omega$.. Therefore $\emptyset$ is the unit for $\shuffle_\lambda$ as claimed.
  \item We prove the commutativity of $\shuffle_\lambda$ by induction on $|F|+|F'|$. If $|F|+|F'|=0$ then $F=F'=\emptyset$ and $F\shuffle_\lambda F'=\emptyset=F'\shuffle_\lambda F$ by definition. Let $N\in\N$ and assume that, for any 
  pair of forest $f,f'$ such that $|f|+|f'|\leq N$ we have $f\shuffle_\lambda f'=f'\shuffle_\lambda f$. Let $F, F'$ be two forests such that $|F|+|F'|=N+1$. We then distinguish three cases:
  \begin{itemize}
   \item If $F=\emptyset$ or $F'=\emptyset$, then $F\shuffle_\lambda F'= F'\shuffle_\lambda F$ since $\emptyset$ is the unit of $\shuffle_\lambda$.
   \item If $F=T=B_+^{a}(f)$ and $F'=T'=B_+^{a'}(f')$ are two nonempty trees, then 
   \begin{equation*}
    T\shuffle_\lambda T' - T'\shuffle_\lambda T = \lambda B_+^{a.a'}(f\shuffle_\lambda f') - \lambda B_+^{a'.a}(f'\shuffle_\lambda f).
   \end{equation*}
   The R.H.S. vanishes by commutativity of $(\Omega,.)$ and the induction hypothesis.
   \item If $F$ or $F'$ is not a tree, we write $F=T_1\cdots T_k$ and $F'=t_1\cdots t_n$ with $k+n\geq3$. Then 
   \begin{align*}
    F\shuffle_\lambda F' & = \frac{1}{kn}\sum_{i=1}^k\sum_{j=1}^n\left((T_i\shuffle_\lambda t_j)T_1\cdots\widehat{T_i}\cdots T_n t_1\cdots\widehat{t_j}\cdots t_k\right) \\
			 & = \frac{1}{kn}\sum_{i=1}^k\sum_{j=1}^n\left((t_j\shuffle_\lambda T_i)t_1\cdots\widehat{t_j}\cdots t_kT_1\cdots\widehat{T_i}\cdots T_n\right)
   \end{align*}
   by the induction hypothesis and the commutativity of the concatenation product of trees. Exchanging the order of the summations we indeed find $F\shuffle_\lambda F' = F'\shuffle_\lambda F$.
  \end{itemize}
  We have treated the three possible cases. Thus $\shuffle_\lambda$ is indeed commutative.
 \end{enumerate}

\end{proof}

\begin{rk} \label{rk:future_research}
 We will focus here on the applications of these new shuffle products of trees and their nonassociativity to the study of AZVs. Linked question, such that the existence of a coproduct associated to these shuffles and the 
 eventual existence of a comodule-bialgebra structure \cite{E-FFM} are interesting questions, but away from the scope of the present work. As such, they are left for further investigations.
%
%
\end{rk}

\begin{rk}
 One can see\footnote{I thank Dominique Manchon for his very useful comments on this point.} that the coefficient $1/kn$ arising in the compatibility with the concatenation product of trees equation of the 
 definition of the shuffle products on trees will prevent associativity. However, neither is the same 
 product defined without these factors associative. If one computes $((T_1\dots T_n)\shuffle_\lambda T)\shuffle_\lambda(t_1\cdots t_m)$, we will see terms containing the forests $(T_i\shuffle_\lambda T)(T_j\shuffle_\lambda t_k)$ 
 appear. Such terms will not be present in products  $(T_1\dots T_n)\shuffle_\lambda (T\shuffle_\lambda(t_1\cdots t_m))$.
 
\end{rk}

Let us end this subsection by an illustration of the nonassocativity of $\shuffle_\lambda$, with $\lambda$ set to $0$ in order to have simpler 
expressions to handle.
\begin{coex}
 Let $\Omega$ be a set. Then an easy computation gives, for any $(a,b,c,d)\in\Omega^4$
 \begin{equation*}
  \left(\tdun{a}\tdun{b}\shuffle~\tdun{c}\right)\shuffle~\tdun{d} = \tdun{a}\tdun{b}\shuffle\left(\tdun{c}\shuffle~\tdun{d}\right) + \frac{1}{2}\left(\tddeux{a}{d} + \tddeux{d}{a}\right)\left(\tddeux{b}{c} + \tddeux{c}{b}\right) + \frac{1}{2}\left(\tddeux{b}{d} + \tddeux{d}{b}\right)\left(\tddeux{a}{c} + \tddeux{c}{a}\right).
 \end{equation*}

\end{coex}

\subsection{Shuffles of trees and Rota-Baxter maps}

Shuffle products on trees are also linked to Rota-Baxter operators, as stated in the following theorem. It is a generalisation to the case of trees of Theorem 2.9 of \cite{CGPZ1}. However its proof differs from the proof of the 
theorem below in two ways. First, the proof in \cite{CGPZ1} uses the universal property of words while purely combinatorial techniques are used here (note that purely combinatorial techniques could also have been used in 
\cite{CGPZ1}). Moreover, the compatibility of the shuffle products on trees with the concatenation product of trees have no equivalence for words. Therefore, the last part of the current proof does not have an counterpart in 
\cite{CGPZ1}.
\begin{thm} \label{thm:branching_shuffle_tree_morphism}
 Let $\lambda\in\R^*$ and $(\Omega,.)$ be a commutative algebra. Let $P:\Omega\longrightarrow\Omega$ be a linear map. The following two statements are equivalents:
 \begin{enumerate}
  \item $P$ is a Rota-Baxter operator of weight $\lambda$;
  \item $\widehat{P}$ is an algebra morphism for the $\lambda$ shuffle product on trees:
  \begin{equation*}
   \widehat{P}(F\shuffle_\lambda F') = \widehat{P}(F)\widehat{P}(F').
  \end{equation*}
 \end{enumerate}
\end{thm}
\begin{proof}
 $\Leftarrow$: as in the proof of Theorem \ref{thm:flattening}, writing $\widehat{P}(F\shuffle_\lambda F') = \widehat{P}(F)\widehat{P}(F')$ for 
 two forests $F$ and $F'$ having only one vertex leads to the Rota-Baxter equation \eqref{eq:RBO}.
 
 $\Rightarrow$: We prove the result by induction on $|F|+|F'|$. For $|F|+|F'|=0$ we have $F=F'=\emptyset$. Then the result holds by definition 
 of $\widehat{P}$.
 
 Assuming the result holds for any $f$, $f'$ such that $|f|+|f'|\leq n$ for some $n\in\N$. Let $F, F'$ be two forests such that $|F|+|F'|=n$. 
 We separate three cases:
 \begin{itemize}
  \item If $F$ or $F'$ is empty, then the result holds since $\emptyset$ is the unit of $\shuffle_\lambda$ and $\widehat{P}(\emptyset):=1_\Omega$.
  \item If $F$ and $F'$ are two nonempty trees, we write $F=B_+^a(f)$ and $F'=B_+^b(f')$. Then by definition of $\shuffle_\lambda$ and 
  linearity of $\widehat{P}$ we have
  \begin{align*}
   \widehat{P}(B_+^a(f)\shuffle_\lambda B_+^b(f')) & = \widehat{P}\left(B_+^a(f\shuffle_\lambda B_+^b(f')\right) + \widehat{P}\left(B_+^b(B_+^a(f)\shuffle_\lambda f')\right) + \lambda\widehat{P}\left(B_+^{ab}(f\shuffle_\lambda f')\right) \\
						   & = P\left(a\widehat{P}(f)\widehat{P}(B_+^b(f'))\right) + P\left(b\widehat{P}(B_+^a(f))\widehat{P}(f')\right) + \lambda P\left(ab\widehat{P}(f)\widehat{P}(f')\right)
  \end{align*}
  where we have used the definition of $\widehat{P}$ and the induction hypothesis. To increase readability, let us write $A:=a\widehat{P}(f)$ 
  and $B:=b\widehat{P}(f')$. Then using the commutativity of $\Omega$ we can write
  \begin{equation*}
   \widehat{P}(B_+^a(f)\shuffle_\lambda B_+^b(f')) = P(AP(B)) + P(P(A)B) + \lambda P(AB) = P(A)P(B)
  \end{equation*}
  since $P$ is a Rota-Baxter operator of weight $\lambda$. Observing that $P(A)=\widehat{P}(F)$ and $P(B)=\widehat{P}(F')$ we obtain the result.
  \item Finally, if $F$ or $F'$ is non-tree forest, we write $F=T_1\cdots T_k$ and $F'=t_1\cdots t_n$ with $k+n\geq3$. Then, using the 
  definitions of $\shuffle_\lambda$ and $\widehat{P}$ (with the convention that an empty product $\prod_{n\in\emptyset}$ is one) we got
  \begin{align*}
   \widehat{P}\left(F\shuffle_\lambda F'\right) & = \frac{1}{kn}\sum_{i=1}^k\sum_{j=1}^n\widehat{P}\left(T_i\shuffle_\lambda t_j\right)\prod_{\substack{p=1\\p\neq i}}^k\widehat{P}(T_p)\prod_{\substack{q=1\\q\neq j}}^k\widehat{P}(t_q) \\
						& = \frac{1}{kn}\sum_{i=1}^k\sum_{j=1}^n\widehat{P}(T_i)\widehat{P}(t_j)\prod_{\substack{p=1\\p\neq i}}^k\widehat{P}(T_p)\prod_{\substack{q=1\\q\neq j}}^k\widehat{P}(t_q) \quad \text{by induction hypothesis} \\
						& = \frac{1}{kn}\sum_{i=1}^k\sum_{j=1}^n\prod_{\substack{p=1}}^k\widehat{P}(T_p)\prod_{\substack{q=1}}^k\widehat{P}(t_q) \\
						& = \widehat{P}(F)\widehat{P}(F')
  \end{align*}
  since $\widehat{P}$ is an algebra morphism for the concatenation product of trees.
 \end{itemize}
 This concludes the induction step and the proof of the Theorem.
\end{proof}

\subsection{Applications to arborified zetas}

In order to apply Theorem \ref{thm:branching_shuffle_tree_morphism} to the case of AZVs we first need the following Lemmas.
\begin{lem} \label{lem:stability_conv_shuffle}
 $(\calf_{\N^*}^{\rm conv},\stuffle)$ (resp. $(\calf_{\N^*}^{\rm conv},\shuffle_{-1})$, resp. $(\calf_{\{x,y\}}^{\rm conv},\shuffle)$) is a
 subalgebra of $(\calf_{\N^*},\stuffle)$ (resp. $(\calf_{\N^*},\shuffle_{-1})$, resp. $(\calf_{\{x,y\}},\shuffle)$).
\end{lem}
\begin{proof}
 Using the fact that $F\shuffle_\lambda F'=\sum_i f_i$ implies $|f_i|\leq |F|+|F'|$, we can easily prove the Lemma by 
 induction on $|F|+|F'|$, which then holds by definition of $\calf_{\N^*}^{\rm conv}$ and $\calf_{\{x,y\}}^{\rm conv}$.
\end{proof}
\begin{lem} \label{lem:sharp_morphism_shuffle}
 For any algebras morphism $P:\Omega_1\longrightarrow\Omega_2$ between two commutative algebras, the lifted map 
 $P^\sharp:\calf_{\Omega_1}\longrightarrow\calf_{\Omega_2}$ is an algebra morphism for the $\lambda$-shuffles of trees, for any value of 
 $\lambda\in\R$.
\end{lem}
\begin{proof}
 This result clearly holds from the definition of $P^\sharp$ and the fact that $P$ is an algebra morphism. However, it is easily proven, once again by induction on the number of vertices of forests, using the fact that 
 $P^\sharp$ is a morphism of operated algebras.
\end{proof}

We can now prove the main result of this Section.
\begin{thm} \label{thm:AZV_alg_morphism_shuffle}
 The map $\zeta^T_\stuffle:\calf_{\N^*}^{\rm conv}\longrightarrow\R$ (resp. $\zeta^{T,\star}_\stuffle:\calf_{\N^*}^{\rm conv}\longrightarrow\R$, resp. $\zeta^T_\shuffle:\calf_{\{x,y\}}^{\rm conv}\longrightarrow\R$) is an 
 algebra morphism for the stuffle (resp. anti-stuffle, resp. shuffle) product on trees.
\end{thm}
\begin{proof}
 Recall that $\zeta^T_\stuffle=\rm{ev}_\infty\circ\widehat\frakS_{-1}\circ\calr^\sharp$ and 
 $\zeta^{T,\star}_\stuffle=\rm{ev}_\infty\circ\widehat\frakS_{0}\circ\calr^\sharp$. By Lemma \ref{lem:sharp_morphism_shuffle} and Theorem 
 \ref{thm:branching_shuffle_tree_morphism} (which can be used since $\frakS_{-1}$ is a Rota-Baxter operator of weight $+1$  and $\frakS_{+1}$ is a Rota-Baxter operator of weight $-1$, 
 as stated in Example \ref{ex:RBmap}), the maps $\widehat\frakS_{-1}\circ\calr^\sharp$ and 
 $\widehat\frakS_{0}\circ\calr^\sharp$ are 
 a morphism for the stuffle product on trees and the anti-stuffle product respectively. The result then follows from Lemma 
 \ref{lem:stability_conv_shuffle} and the fact that ${\rm ev}_\infty$ is an algebra morphism.
 
 The case of $\zeta_\shuffle^T={\rm ev}_1\circ Li^T\circ\calr_{\{x,y\}}^\sharp$ is proven exactly in the same fashion, using the fact that the 
 map $Li^T$ is obtained from the branching of the integration map $\cali$, which is a Rota-Baxter operator of weight $0$.
\end{proof}
\begin{rk}
 Since $\iota_\Omega(w\shuffle_\lambda w')=\iota_\Omega(w)\shuffle_\lambda\iota_\Omega(w')$ and since AZVs restricted to ladder trees coincide with MZVs, this result also implies that $\zeta_\stuffle$ (resp. 
 $\zeta_\stuffle^\star$, resp. $\zeta_\shuffle$) is an algebra morphism for the stuffle (resp. anti-stuffle, resp. shuffle) product. 
\end{rk}
As a concluding observation for this section, let us show that Theorem \ref{thm:AZV_alg_morphism_shuffle} induces relations amongst 
AZVs, namely that the images of the associators of the shuffle products on trees lie in the kernels of the AZVs.

Recall that the {\bf associator} $[.,.,.]_{\stuffle}$  of the stuffle product is defined by
 \begin{align*}
  [.,.,.]_{\stuffle}:\calf_{\N^*}\times\calf_{\N^*}\times\calf_{\N^*} & \longmapsto \calf_{\N^*} \\
  (F_1,F_2,F_3) & \longrightarrow (F_1\stuffle F_2)\stuffle F_3-F_1\stuffle(F_2\stuffle F_3).
 \end{align*}
 We similarly define associators for the shuffle and anti-stuffle products.

\begin{coro} \label{coro:relation_nonasso}
 For any $\lambda\in\{-1,0,1\}$,
 \begin{equation*}
  {\rm Im}\left([.,.,.]_{\shuffle_\lambda}\right)\subseteq{\rm Ker}\left(\zeta^T_{\shuffle_\lambda}\right).
 \end{equation*}
More specifically, for any  $(F_1,F_2,F_3)\in(\calf_{\N^*}^{\rm conv})^3$ and $(f_1,f_2,f_3)\in(\calf_{\{x,y\}}^{\rm conv})^3$ we have
 \begin{align*}
  & (F_1\stuffle F_2)\stuffle F_3 - F_1\stuffle (F_2\stuffle F_3) \in {\rm Ker}(\zeta_\stuffle^T) \\
  & (F_1\shuffle_{-1} F_2)\shuffle_{-1} F_3 - F_1\shuffle_{-1} (F_2\shuffle_{-1} F_3) \in {\rm Ker}(\zeta_\stuffle^{T,\star}) \\
  & (f_1\shuffle f_2)\shuffle f_3 - f_1\shuffle (f_2\shuffle f_3) \in {\rm Ker}(\zeta_\shuffle^{T}).
 \end{align*}

\end{coro}
\begin{proof}
 This result is a direct consequence of Theorem \ref{thm:AZV_alg_morphism_shuffle} together with the associativity of the usual product on $\R$.
\end{proof}

\begin{rk}
 These relations between AZVs have non equivalent among MZVs. However, it shall not be expected that it will induce relations among MZVs that cannot be obtained from the usual regularised double shuffle relations. 
 Indeed, there are 
 more AZVs of a given weight than MZVs of the same weight, but all AZVs can be written as MZVs, according to Theorems \ref{thm:main_result_stuffle} and \ref{thm:main_result_shuffle}. 
 Therefore we should indeed find more relations amongst AZVs. This indicates that the nonassocativity of the shuffle products of trees encodes the nonlinearity\footnote{in the sense that 
 not all trees are ladder trees} of trees.
\end{rk}

\begin{ex}
 We illustrate the above Remark by a simple example. Computing $(\tdun{2}\tdun{2}\stuffle\tdun{2})\stuffle\tdun{2}$ and $\tdun{2}\tdun{2}\stuffle(\tdun{2}\stuffle\tdun{2})$, we find
 \begin{align*}
  & \zeta_\stuffle^T\Big((\tdun{2}\tdun{2}\stuffle\tdun{2})\stuffle\tdun{2} - \tdun{2}\tdun{2}\stuffle(\tdun{2}\stuffle\tdun{2})\Big) = 0 \\
  \Longleftrightarrow & \Big[6\zeta(2,2,2) + 3\zeta(2,4) + 3\zeta(4,2) + \zeta(6)\Big]\zeta(2) = \Big[2\zeta(2,2)+\zeta(4)\Big]^2;
 \end{align*}
 which can indeed be shown using the stuffle product on words and checked using e.g. the online calculator \cite{BLI}.

\end{ex}

To conclude this Section, let us emphasize that the shuffle products on trees, while relevant to the study of AZVs, also lead to new and exciting questions, some of which were pointed out in Remark \ref{rk:future_research}.
We have opted to not treat them here, as they would carry us away from AZVs that are the main concern of this paper and are left for future work.

\section*{Appendix: Further relations amongst arborified zetas} \label{section:further}

\addcontentsline{toc}{section}{Appendix: Further relations amongst arborified zetas}

\renewcommand\thesection{\Alph{section}}

\setcounter{section}{1}

\setcounter{subsection}{0}

\setcounter{thm}{0}

\subsection{Branched binarisation map}

We aim to generalise the map \eqref{eq:binarisation_map} to trees, that is to build a map
$\fraks^T :\calf_{\N^*}\longrightarrow\calf_{\{x,y\}}$ which coincide with $\iota_{\{x,y\}}\circ\fraks$ when restricted to ladder trees.
In order to do so, once again we use  the universal property of trees given by \ref{thm:univ_prop_tree}.
\begin{defn} \label{defn:branched_bin_map}
 Let $\beta:\N^*\times\calf_{\{x,y\}}\longrightarrow\calf_{\{x,y\}}$ defined by
 \begin{align*}
  \beta(1,F) := B_+^y(F) \\
  \beta(n,F) := \left(B_+^x\right)^{n-1}\left(B_+^y(F)\right)
 \end{align*}
 for any $n\geq2$. The {\bf branched binarisation map} is the morphism of operated algebras 
 $\fraks^T :\calf_{\N^*}\longrightarrow\calf_{\{x,y\}}$ whose existence and uniqueness is given by Theorem \ref{thm:univ_prop_tree}.
\end{defn}
\begin{ex} Here are some example of the action of the binarisation map:
 \begin{equation*}
  \fraks^T(\tdun{1}) = \tdun{y} \qquad  \fraks^T(\tdun{2}) = \tddeux{$x$}{$y$} \qquad \fraks^T\left(\tdtroisun{1}{1}{2}\right) = \tdquatredeux{$y$}{$y$}{$x$}{$y$} \qquad
  \fraks^T\left(\tdtroisun{$2$}{$1$}{$2$}\right) = \tcinqonze{$x$}{$y$}{$x$}{$y$}{$y$}.
 \end{equation*}

\end{ex}
Now we state a simple lemma relating convergent forests in $\calf_{\N^*}$ and in $\calf_{\{x,y\}}$.
\begin{lem} \label{lem:bin_map_properties}
The branched binarisation map maps convergent forests to convergent forests:
 \begin{equation*}
  \fraks^T\left(\calf_{\N^*}^{\rm conv}\right) = \calf_{\{x,y\}}^{\rm conv}.
 \end{equation*}
 Furthermore, $\fraks^T$ is a bijection.
\end{lem}
\begin{proof}
 By definition of the operation $\beta$, if $F\in\calf_{\{x,y\}}$ is in the image of $\beta$, then $F$ is semiconvergent. Therefore $\fraks^T\left(\calf_{\N^*}\right) \subseteq \calf_{\{x,y\}}^{\rm semi} \subseteq \calf_{\{x,y\}}$.
 Thus we only need to prove that the image of a convergent forest has its roots decorated by $x$s only.
 
 Let $T\in\calf_{\N^*}^{\rm conv}$ be a convergent tree. If $T=\emptyset$ then $\fraks^T(\emptyset)=\emptyset\in\calf_{\{x,y\}}^{\rm conv}$ by definition of $\calf_{\{x,y\}}^{\rm conv}$. 
 
 If $T\neq \emptyset$ then it exists a forest $F$ such that $T=B_+^p(F)$ with $p\geq2$. Then by definition of $\fraks^T$ we have
 \begin{equation*}
  \fraks^T(T) = \left(B_+^x\right)^{p-1}\left(B_+^y(\fraks^T(F))\right)
 \end{equation*}
 which lies in $\calf_{\{x,y\}}^{\rm conv}$ since $p-1\geq1$.
 
 Let $F\in\calf_{\N^*}^{\rm conv}$ be a convergent forest. Then we have $F=T_1\cdots T_k$ with $T_i\in\calf_{\N^*}^{\rm conv}$ by definition 
 of $\calf_{\N^*}^{\rm conv}$. Then 
 \begin{equation*}
  \fraks^T(F) = \fraks^T(T_1)\cdots\fraks^T(T_1)\in\calf_{\{x,y\}}^{\rm conv}
 \end{equation*}
 by definition of $\calf_{\{x,y\}}^{\rm conv}$. Therefore $\fraks^T\left(\calf_{\N^*}^{\rm conv}\right) \subseteq \calf_{\{x,y\}}^{\rm conv}$.
 
 The bijectivity of $\fraks^T$ is also shown by induction, using $|\fraks^T(F)| = ||F||$. The same argument on $(\fraks^T)^{-1}$ allows to show that 
 $(\fraks^T)^{-1}\left(\calf_{\{x,y\}}^{\rm conv}\right) \subseteq \calf_{\N^*}^{\rm conv}$; concluding the proof.
\end{proof}

Recall that a branching vertex is a vertex that has strictly more than one direct successor. This concept will be 
of importance when relating shuffle and stuffle arborified zeta values through the branched binarisation map.

Furthermore, in order to lighten the notations, we will write 
\begin{equation*}
 B_+^{\omega_1\cdots\omega_k} := B_+^{\omega_1}\circ\cdots\circ B_+^{\omega_k}.
\end{equation*}
\begin{thm} \label{thm:relation_shuffle_stuffle}
 For any convergent forest $F\in\calf_{\N^*}^{\rm conv}$ we have
 \begin{equation*}
  \zeta^T_{\shuffle}(\fraks^T(F)) \leq \zeta^T_{\stuffle}(F).
 \end{equation*}
 Furthermore, the inequality is an equality if, and only if, $F$ has no branching vertex (i.e. $F$ is the empty tree or $F=l_1\cdots l_k$ with 
 $l_i$ being ladder trees).
\end{thm}
\begin{proof}
 \begin{itemize}
  \item If $F=\emptyset$, then $\fraks^T(F)=\emptyset$ and the result holds by construction.
  \item If $F=l_1\cdots l_k$ with $l_i$ being ladder trees, then the result follows from the classical property \eqref{eq:shuffle_stuffle_words} of the binarisation map together with the 
  facts that $\zeta^T_\shuffle$ and $\zeta^T_\stuffle$ are algebra morphisms for the concatenation product of trees, $\zeta_\stuffle(w) =  \zeta^T_\stuffle(\iota_{\N^*}(w))$ and 
  $\zeta_\shuffle(w) =  \zeta^T_\shuffle(\iota_{\{x,y\}}(w))$.
  \item We will work out explicitly that for a tree $T$ with exactly one branching vertex, one has
  \begin{equation*}
   \zeta^T_{\shuffle}(\fraks^T(T)) < \zeta^T_{\stuffle}(T).
  \end{equation*}   
  Let $T$ be a tree with $l$ leaves but only one branching vertex. We have
  \begin{equation*}
   T = B_+^{p_1\cdots p_{k}}\left(B_+^{p^1_1\cdots p^1_{k_1}}(\emptyset)\cdots B_+^{p^l_1\cdots p^l_{k_l}}(\emptyset)\right)
  \end{equation*}
  with $p_1\geq2$. Then 
  \begin{equation*}
   \fraks^T(T) = B_+^{\fraks(p_1\cdots p_{k})}\left(B_+^{\fraks(p^1_1\cdots p^1_{k_1})}(\emptyset)\cdots B_+^{\fraks(p^l_1\cdots p^l_{k_l})}(\emptyset)\right).
  \end{equation*}
  For any $i\in\{1,\cdots,l\}$, let $b_i:=B_+^{\fraks(p^i_1\cdots p^i_{k_i})}(\emptyset)$. The $b_i$s are ladder trees in $\calf_{\{x,y\}}$. Then, by definition 
  we have
  \begin{equation*}
   \zeta_\shuffle^T(\fraks^T(T)) = \int_{1\geq t_1\geq \cdots \geq t_{|p_1|+\dots+|p_k|-1}\geq 0}\prod_{i=1}^k\left(\prod_{j=1}^{p_i-1}\frac{dt_{p_1+\cdots+p_{i-1}+j}}{t_{p_1+\cdots+p_{i-1}+j}}\right)\frac{dt_{p_1+\cdots+p_{i}}}{1-t_{p_1+\cdots+p_{i}}}\prod_{r=1}^l\int_0^{t_{|p_1|+\dots+|p_k|-1}} Li^T_{b_r}(z_r)dz_r. 
  \end{equation*}
  Using the standard trick of expanding in series
  \begin{equation*}
   \frac{1}{1-t} = \sum_{n=0}^\infty t^n
  \end{equation*}
  and exchanging series and integral, we can write the last integrations as series:
  \begin{equation*}
   \int_0^{t_{|p_1|+\dots+|p_k|-1}} Li^T_{b_r}(z_r)dz_r = \sum_{n_1>\cdots n_{k_r}>0}^\infty \frac{(t_{|p_1|+\dots+|p_k|-1})^{n_1}}{(n_1)^{p_1^r} \cdots (n_{k_r})^{p_{k_r}^r}}.
  \end{equation*}
  Thus we have 
  \begin{equation*}
   \prod_{r=1}^l\left(\int_0^{t_{|p_1|+\dots+|p_k|-1}} Li^T_{b_r}(z_r)dz_r\right) = \sum_{\substack{n_1^1>\cdots n_{k_1}^1>0 \\ \cdots \\ n_1^l>\cdots n_{k_l}^l>0}}^\infty \frac{(t_{|p_1|+\dots+|p_k|-1})^{\sum_{r=1}^l n_1^r}}{(n_1^1)^{p_1^1} \cdots (n_{k_1}^1)^{p_{k_1}^1} \cdots (n_1^1)^{p_1^1} \cdots (n_{k_l}^l)^{p_{k_l}^l}}.
  \end{equation*}
  Using the series expansion and exchanging series and integrals allow to perform the next $p_k$ integrations. They produce a new series, 
  for a parameter 
  $n>\sum_{r=1}^l n_1^r$. This series is associated to the branching vertex of $T\in\calf_{\N^*}^{\rm conv}$.
  
  On the other hand, the series associated to the branching vertex of $T\in\calf_{\N^*}^{\rm conv}$ in $\zeta^T_\stuffle(T)$ is over a parameter 
  $n>\min\{n_1^1,\cdots,n_1^l\}$. We therefore obtain that $\zeta^T_\stuffle(T) - \zeta^T_\shuffle(\fraks^T(T))$ is a convergent series of positive terms, therefore 
  a positive number.
  
  This shows that a branching vertex in a tree $T$ induces a loss when writing $\zeta^T_\shuffle(\fraks^T(T))$ in terms of series. The same argument 
  holds for the more general case of a tree with an arbitrary number of branching vertices. We therefore obtain 
  $\zeta^T_{\shuffle}(\fraks^T(T)) < \zeta^T_{\stuffle}(T)$ for any non-ladder tree $T$. 
  \item The result for a forest follows from the previous point and the fact that $\zeta^T_\shuffle$ and $\zeta^T_\stuffle$ are algebra morphisms for the 
  concatenation of trees (Proposition \ref{prop:stuffle_alg_mor} and Theorem \ref{thm:main_result_shuffle}).
 \end{itemize}

\end{proof}

\subsection{Hoffman's relations for branched zetas}

In order to find the branched equivalent of Hoffman's regularisation relations, it is useful to recall that the concatenation product of trees can be seen as lift to trees of both the shuffle and stuffle products. 
The latter only differ in the way one goes back to words, that is to say by a choice of the 
flattening map. Therefore the most naive candidate for this relations is that $\fraks^T(\tdun{1} F) - \fraks^T(\tdun{1})\fraks^T(F)$ is convergent 
for any convergent forest $F$ and lies in the kernel of $\zeta_\shuffle^T$.

However, this statement is trivially true and does not impose any new relation on the algebra spanned by branched zetas. Indeed, since $\fraks^T$ is an 
algebra morphism we have $\fraks^T(\tdun{1} F) - \fraks^T(\tdun{1})\fraks^T(F)=0$\footnote{notice that $0$ and $\emptyset$ are distinct elements of our algebra}. 
This observation is a consequence of the fact that the distinction 
between the stuffle and shuffle products can only be made at the level of words, reached through flattening maps. Consequently, a more relevant quantity 
to study is 
\begin{equation*}
 \fraks\left(fl_1\left(\tdun{1} F\right)\right) - fl_0\left(\fraks^T\left(\tdun{1} F\right)\right)
\end{equation*}
for any convergent forest $F$.
\begin{prop} \label{prop:branched_Hoffman}
 For any convergent forest $F\in\calf_{\N^*}^{\rm conv}$, 
 \begin{equation*}
  \fraks\left(fl_1\left(\tdun{1} F\right)\right) - fl_0\left(\fraks^T\left(\tdun{1} F\right)\right)
 \end{equation*}
 lies in the algebra of convergent words if, and only if, $F$ is empty or $F=l_1\cdots l_k$ with the $l_i$ ladder trees. In this case 
 \begin{equation*}
  \fraks\left(fl_1\left(\tdun{1} l_1\cdots l_k\right)\right) - fl_0\left(\fraks^T\left(\tdun{1} l_1\cdots l_k\right)\right) \in \Ker(\zeta_\shuffle).
 \end{equation*}
\end{prop}
\begin{proof}
 \begin{itemize}
  \item If $F$ is empty, then $\fraks\left(fl_1\left(\tdun{1} F\right)\right) - fl_0\left(\fraks^T\left(\tdun{1} F\right)\right)=0$ and the 
  result trivially holds.
  \item If $F=l_1\cdots l_k$ with the $l_i$ ladder trees then, by definition of the flattening maps we have
  \begin{equation*}
  \fraks\left(fl_1\left(\tdun{1} F\right)\right) - fl_0\left(\fraks^T\left(\tdun{1} F\right)\right) = \fraks\left((y)\stuffle w_1\stuffle\cdots\stuffle w_k\right) - (y)\shuffle\fraks(w_1)\shuffle\cdots\shuffle\fraks(w_k)
 \end{equation*}
 with $w_i:=\iota_{\N^*}^{-1}(l_i)$. Then the result holds by Hoffman's regularisation relations \eqref{eq:Hoffman_reg_rel}.
 \item Let $T$ be a non ladder tree, and $v$ a branching vertex of $T$ with decoration $p_1$. Let $v'$ and $v''$ be two direct successors of $V$ with decorations $p_2$ and $p_3$ respectively. Then it exists two 
 eventually empty words $w$ and $w'$ such that
 \begin{equation*}
  fl_1(\tdun{1} T) = (1)\sqcup w\sqcup(p_1[p_2+p_3])\sqcup w' + X
 \end{equation*}
 with $X$ a finite linear combination of words written in the alphabet $\N^*$.
 
 By definition of $fl_0$, $\fraks\left((1)\sqcup w\sqcup(p_1[p_2+p_3])\sqcup w'\right)$ will not show up in $fl_0\left(\fraks^T\left(\tdun{1} F\right)\right)$. Thus we obtain
 \begin{equation*}
  \fraks\left(fl_1\left(\tdun{1} F\right)\right) - fl_0\left(\fraks^T\left(\tdun{1} F\right)\right) = \fraks\left((1)\sqcup w\sqcup(p_1[p_2+p_3])\sqcup w'\right) + Y
 \end{equation*}
 with $Y$ a finite linear combination of words written in the alphabet $\{x,y\}$ without the divergent tree 
 $\fraks\left((1)\sqcup w\sqcup(p_1[p_2+p_3])\sqcup w'\right)$. This implies the result for any non-ladder tree.
 
 For a forest with at least one branching point, the result follows from the fact that $\fraks^T$ and the flattening maps are algebras morphisms and the previous discussion on non ladder trees.
 \end{itemize}
\end{proof}

This result was derived from the picture that the shuffle and stuffle products of words are lifted to trees to the concatenation of trees. We 
have seen in Section \ref{section:shuffle_tree} that one can instead define shuffle and stuffle products on trees. This leads us to an alternative  possible 
generalisation of Hoffman's relations  \eqref{eq:Hoffman_reg_rel}; namely that

\begin{equation*}
 \tdun{1}\stuffle ~T - (\fraks^T)^{-1}\left(\tdun{y}\shuffle ~\fraks^T(T)\right) \quad \text{and} \quad \fraks^T\left(\tdun{1}\stuffle ~T\right) - \tdun{y}\shuffle ~\fraks^T(T)
\end{equation*}
lie in $\calf_{\N^*}^{\rm conv}$
for any convergent tree $T\in\calf_{\N^*}^{\rm conv}$. And indeed 
\begin{prop} \label{prop:conv_Hoffman_trees}
 For any convergent forest $F\in\calf_{\N^*}{\rm conv}$, one has
 \begin{align*}
  & \tdun{1}\stuffle ~F - (\fraks^T)^{-1}\left(\tdun{y}\shuffle ~\fraks^T(F)\right) \in \calf_{\N^*}^{\rm conv}; \\
  & \fraks^T\left(\tdun{1}\stuffle ~F\right) - \tdun{y}\shuffle ~\fraks^T(F) \in \calf_{\{x,y\}}^{\rm conv}.
 \end{align*}
\end{prop}
\begin{proof}
 First, notice that by Lemma \ref{lem:bin_map_properties} the two statements are equivalent. We therefore only prove the second one.
 
 For a convergent forest $F=T_1\cdots T_k$ , we have $\fraks^T(F) = \fraks^T(T_1)\cdots\fraks^T(T_k)$. Then we have, with obvious 
 notations:
 \begin{equation*}
  \tdun{1}\stuffle~F = \frac{1}{k}\sum_{i=1}^k\left(B_+^1(T_i) + X_{i,F}\right)F\setminus T_i
 \end{equation*}
 for some finite sum of trees $X_{i,F}\in\calf_{\N^*}{\rm conv}$. Therefore we have
 \begin{equation*}
  \fraks^T\left(\tdun{1}\stuffle~F\right) = \frac{1}{k}\sum_{i=1}^k\left(B_+^y(\fraks^T(T_i)) + \fraks^T(X_{i,F})\right)\fraks^T(F\setminus T_i).
 \end{equation*}
 On the other hand we have
 \begin{equation*}
  \tdun{y}\shuffle~\fraks^T(F) = \frac{1}{k}\sum_{i=1}^k\left(B_+^y(\fraks^T(T_i)) + Y_{i,F}\right)\fraks^T(F\setminus T_i)
 \end{equation*}
 for some $Y_{i,F}\in\calf_{\{x,y\}}^{\rm conv}$. Taking the difference of these two quantities, one obtain the result, since by 
 Lemma \ref{lem:bin_map_properties} the branched binarisation map $\fraks^T$ maps convergent forests to convergent forests.
\end{proof}

While this result might give us hope, one should expect Theorem \ref{thm:relation_shuffle_stuffle} to prevent the quantities 
$\fraks^T\left(\tdun{1}\stuffle ~F\right) - \tdun{y}\shuffle ~\fraks^T(F)$ from lying in the kernel of $\zeta^T_\shuffle$. And indeed, one 
finds, after a long yet straightforward computation, that
\begin{equation*}
 \zeta_\shuffle^T\left(\fraks^T\left(\tdun{1}\stuffle\tdtroisun{2}{1}{1}\right) - \tdun{$y$}\shuffle\tdquatrequatre{$x$}{$y$}{$y$}{$y$}\right) = \zeta(2,3)+\zeta(3,2) > 0.
\end{equation*}
The precise characterisation of quantities of the form $\fraks^T\left(\tdun{1}\stuffle ~F\right) - \tdun{y}\shuffle ~\fraks^T(F)$ are 
left for future investigations. For now, let us write that the conclusion of this Section is that branching vertices, at least with 
the current definitions of branched zeta values, induce an important change when lifting the properties of  {multiple zeta} values to arborified zeta values. \\

{\bf Acknowledgments:} The author thanks Dominique Manchon for initial discussions that kick-started this project and further fruitful 
interactions. I also thank Lo\"ic Foissy and Li Guo for helpful  
suggestions regarding the literature and Sylvie Paycha for support, encouragements and advice. Finally, many thanks Dominique Manchon, Sylvie Paycha and Li Guo for comments on preliminary versions of this paper.

\end{document}